\documentclass{amsart}    
\usepackage{amsrefs}

\usepackage{amssymb,enumerate}

\usepackage{mathrsfs}
\usepackage{mathtools} 
\usepackage{bigstrut}
\usepackage{tikz}
\usepackage{tikz-cd}
\usetikzlibrary{patterns}
\usetikzlibrary{matrix,arrows,positioning, decorations.pathmorphing,shapes.geometric}
\usepackage{calc}

\pgfmathsetmacro{\nodebasesize}{1} % A node with a value of one will have this diameter
\pgfmathsetmacro{\nodeinnersep}{0.1}

\newcommand{\posnode}[3]{% position, name, options, value, label
    \pgfmathsetmacro{\minimalwidth}{sqrt(#3*\nodebasesize)}
    \node[%#3,
minimum width=\minimalwidth*.5cm,inner sep=\nodeinnersep*.5cm,circle,draw] (#2) at (#1) {};  %{#5};
}

           \newcommand{\labnode}[5]{% position, name, label options +fill options,  value label
    \pgfmathsetmacro{\minimalwidth}{sqrt(#4*\nodebasesize)}
    \node[#3,#5,minimum width=\minimalwidth*.5cm,inner sep=\nodeinnersep*.5cm,circle,draw] (#2) at (#1) {};
}

\usepackage{amsmath}
\usepackage{amsrefs}
\usepackage{verbatim}
\usepackage{setspace}
\newtheorem{thm}{Theorem}[section]
\newtheorem{lem}[thm]{Lemma}

\newtheorem{prop}[thm]{Proposition}

\theoremstyle{definition}

\newtheorem{ex}[thm]{Example}

\newenvironment{newlist}
   {\begin{list}{}{\setlength{\labelsep}{0.25cm}
\setlength{\itemsep}{-0.1cm}
\setlength{\topsep}{0.1cm}
                   \setlength{\labelwidth}{0.65cm}
                      \setlength{\leftmargin}{0.9cm}}}
   {\end{list}}

\newenvironment{widenewlist}
   {\begin{list}{}{\setlength{\labelsep}{0.2cm}
                   \setlength{\labelwidth}{1.2cm}
                      \setlength{\leftmargin}{1.3cm}}}
   {\end{list}}

\newcommand{\twiddle}[1]{\smash{\underset{\raise.375ex\hbox{$\smash\sim$}}      {#1}}\vphantom{\underline{#1}}}

\newcommand{\twiddleodd}[1]{
\smash{\underset{\lower.5ex\hbox{$\smash{\widetilde{\phantom{mm}}}
$}
}{#1%\Zed_{2n+1}
}}
}
\newcommand{\twiddleeven}[1]{
\smash{\underset{\lower.65ex\hbox{$\smash{
\widetilde{\phantom{rr\!\!r}}}
$}
}{#1%\Zed_{2n}
}}}
\newcommand{\ov}[1]{\overline{#1}}

\newcommand{\twiddlek}{\smash{\underset{\lower.4ex\hbox{$\smash{
\widetilde{\phantom{w\!\!u}}}
$}
}{\Zed_{k}}}}

\newcommand{\grk}{\graph k}
\newcommand{\grh}{\graph h}
\DeclareMathOperator{\HOM}{Hom}
\newcommand{\PHOM}{\HOM_{\text{\rm p}}}
\newcommand{\PHOMt}{\HOM_{\text{\rm p}}^{\text{\bf t}}}

\newcommand{\delp}{\alpha^+}  %{\delta^+}
\newcommand{\delm}{\alpha^-}     %    {\delta^-}

\newcommand{\delpv}{\delta^+}  %{\delta^+}
\newcommand{\delmv}{\delta^-}     %    {\delta^-}
\newcommand{\dotbigcup}{\overset{\raise.6ex\hbox{$\smash\cdot$}\,}{\smash\bigcup\,}} 
\newcommand{\du}{\smash{\,\cup\kern-0.45em\raisebox{1ex}{$\cdot$}}\,\,}

\newcommand{\class}[1]{\mathcal{#1}}
\newcommand{\cat}[1]{\boldsymbol{\mathscr{#1}}}  
\newcommand{\alg}[1]{\mathbf{#1}}
\newcommand{\fnt}[1]{\mathsf{#1}}
\newcommand{\E}{\fnt{E}}
\newcommand{\D}{\fnt{D}}
\newcommand{\ope}[1]{\mathbb{#1}}
\newcommand{\defn}[1]{\emph{#1}}
\newcommand{\DM}{\class{DM}}
\newcommand{\CCD}{\class{D}}

\newcommand{\PEW}[1]{\End_{\text{\rm p}}^{\text{\bf t}} (\W_{#1})} 

\newcommand{\CV}{\class{V}}
\newcommand{\CQ}{\class{Q}}

\newcommand{\Kl}{\cat{K\!\! A}}
\newcommand{\Kla}{\Kl_u}

\newcommand{\A}{\alg{A}}
\newcommand{\B}{\alg{B}}

\newcommand{\M}{\alg{M}}

\renewcommand{\P}{\alg{P}}
\renewcommand{\S}{\alg{S}}
\newcommand{\T}{\alg{T}}
\newcommand{\Q}{\alg{Q}}
\newcommand{\X}{\alg{X}}
\newcommand{\Y}{\alg{Y}}

\newcommand{\W}{\alg{W}}

\newcommand{\Tp}{\mathcal{T}}
\newcommand{\MT}{\twiddle{\M}}
\newcommand{\Zed}{\alg{Z}}
\newcommand{\two}{\boldsymbol 2}
\newcommand{\SA}{\cat{S\!A}}  
\newcommand{\CA}{\cat{A}}
\newcommand{\SM}{\cat{S\!\!M}} 

\newcommand{\CVM}[1]{\CV(\W_{#1})}
\newcommand{\CVA}[1]{\CV(\Zed_{#1})}
\newcommand{\CQM}[1]{\CQ(\W_{#1})}

\newcommand{\CX}{\cat{X}}
\newcommand{\CY}{\cat{Y}}  
\newcommand{\CM}{\mathcal{M}}
\newcommand{\CMT}{\twiddleeven{\CM}}

\newcommand{\F}{\alg{F}}  

\newcommand{\w}{\omega}
\renewcommand{\epsilon}{\varepsilon}

\newcommand{\sub}[1]{_{_{\kern-.9pt{\scriptstyle #1}}}}
\newcommand{\medsub}[2]{#1\lower0.6ex\hbox{$\scriptstyle{#2}$}}
\newcommand{\eA}[1]{\medsub e {\kern-0.75pt\A\kern-0.75pt}(#1)}
\newcommand{\esub}[1]{\medsub e {\kern-0.75pt #1 \kern-0.75pt}}
\newcommand{\epsub}[1]{\medsub \varepsilon {\kern-1.25pt #1}}
\newcommand{\esubA}{\medsub e {\kern-0.75pt\A\kern-0.75pt}}

\DeclareMathOperator{\ISP}{\ope{ISP}}
\DeclareMathOperator{\IScP}{\ope{IS}_c\ope{P}^{+}}
 \DeclareMathOperator{\HSP}{\ope{HSP}}
 \DeclareMathOperator{\si}{\rm si}
 \DeclareMathOperator{\dom}{\rm dom}
\DeclareMathOperator{\graph}{\rm graph}

\newcommand{\id}{\mathop{}\mathopen{}\mathrm{id}} 
\DeclareMathOperator{\img}{im}

\DeclareMathOperator{\End}{End}  
\newcommand{\Endt}{\End^{\mathbf t}}
\newcommand{\PEZ}[1]{\End_{\text{\rm p}} (\Zed_{#1})} 
\newcommand{\three}{\boldsymbol 3}

\renewcommand{\leq}{\leqslant}
\renewcommand{\geq}{\geqslant}

\begin{document}

\title%[Sugihara Algebras and Monoids]
{Sugihara algebras and Sugihara monoids: \hbox{multisorted dualities}}

%%%%%%%%%%%%%%%% The case of two authors
%%%%%%%%%%%%%%%%%%% and at least two lines t
\author{Leonardo M.  Cabrer}
\address{}
\email{lmcabrer@yahoo.com.ar}
\author{Hilary A. Priestley}
\address{Mathematical Institute, University of Oxford,\\
 Radcliffe Observatory Quarter,\\
Oxford, OX2 6GG, United Kingdom}
\email{hap@maths.ox.ac.uk}

\begin{abstract}
The authors developed in a recent paper natural dualities for finitely generated quasivarieties of Sugihara algebras.  They  thereby  identified
the admissibility algebras for these quasivarieties which, via the Test Spaces Method devised by Cabrer \textit{et al.}, give access to 
a viable method for studying admissible rules within relevance logic, specifically for  extensions of the deductive system $R$-mingle.

This paper builds on the work already done on the theory of natural dualities for Sugihara algebras.  Its purpose is to 
provide an integrated suite of multisorted duality theorems of a uniform type,  
encompassing  finitely generated 
quasivarieties and varieties of both Sugihara algebras and Sugihara
monoids, and  embracing both the odd and the even cases.  
The overarching theoretical 
 framework  of multisorted duality theory  developed here leads on to amenable
representations of free algebras.  More widely, it
%also has the potential
provides a springboard to further applications.  
%\textit{inter alia}, 
%easy access to Priestley-style dualities for the quasivarieties  and varieties of interest and to 
%amenable representations of free algebras.
%
 \end{abstract}

 %\begin{keyword}  

\subjclass[2010]{Primary:  03G25;  %Other algebras related to logic 
Secondary: 
03B47,  %entailment etc 
%%08C15,  %quasivarieties
08C20  %natural duality  ? make this primary   NO
06D50}  
%%\end{keyword}C29

%\end{frontmatter}

\maketitle

\section{Introduction} \label{sec:intro}

Sugihara monoids and algebras, particularly in the odd case,
have attracted 
 interest  on a number of fronts.     
This has stemmed in part 
 from  investigations of  the rich class of residuated lattices
and of  associated substructural logics.  See for example the
introductory  survey \cite{BT02} and the comprehensive monograph \cite{GJKO} for the general theory and context and
\cites{FG17, GR12, GR15, OR07, MRW18}  for indications of where the special class of  Sugihara monoids fits into a wider algebraic framework.    
%%%%%

From the perspective of logic, the motivation for studying Sugihara monoids originates in the work of Dunn \cite{Du70} and Anderson and Belnap \cite{AB75}; see also the contextual comments in \cite{OR07}. 
Sugihara algebras provide complete
algebraic semantics for $\text{RM}$, the algebraizable deductive system $R$-mingle. 
 Sugihara monoids  include a constant in their algebraic language  so as to  model $
\text{RM}^{\mathbf{t}}$,  \textit{viz.} $R$-mingle afforced with Ackermann's constant.  

%hp1301 insert 
%lc1801: I do like both paragraphs, however they sound a bit repetitive. Can we merge them in one? 
%hp2201  Will ponder.  Not done yet
Our introductory remarks above indicate that effective mathematical tools for studying Sugihara algebras and monoids 
would be of interest and value both as regards these algebraic
structures \textit{per se} and   for the algebraic and relational 
semantics for related deductive systems.  Our objective in this paper is to enlarge  the armoury of such tools, as a basis for further theory.  
Our principal theorems provide, within a uniform framework,  a multisorted natural duality 
for every variety generated by a finite subdirectly algebra in either 
$\SA$ or~$\SM$.   Here  the use of dual categories 
based on multisorted topological relational structures is central to 
our approach.  
Our  longer-term aim is for the present work to lead on to 
 a full exploration of alternative relational models, and to a structural analysis of  the  classes $\SA$ of Sugihara algebras and 
$\SM$ of Sugihara monoids based on these models;   we emphasise
that in this programme we treat $\SA$ and $\SM$ together, and on an equal footing.
We  cannot achieve our objectives within a single paper.
Accordingly, 
the material on multisorted dualities we present here, and specifically our theorems in 
Sections~\ref{sec:alg-multi} and~\ref{sec:mon-multi}, may be seen in part as a stepping stone along the way to  our ultimate goal.  
Paper~\cite{CPfree}, devoted to free algebras, provides a 
foretaste of what our new methods can achieve.

%Our work on multisorted dualities 
% has the longer-term goal of leading
%to a full exploration of alternative relational models for
%finitely generated quasivarieties and varieties of  $\SA$ and of~$\SM$, and to a structural analysis based on these models.
%(Our appendix hints at such possibilities.) 
%However this cannot be achieved within a single paper.  
%Accordingly,  our objectives in this paper are more limited, and the material we present, and specifically our theorems in 
%Sections~\ref{sec:alg-multi} and~\ref{sec:mon-multi}, may be seen in part as stepping stones along the way to  our ultimate goal.
%

%%%%%%%%%%%%%%%%%%%%%%%%%%%%%%%%%%%%%%% 

Our %previous 
paper \cite{CPsug}  was motivated by 
a very specific problem:  
 to find a computationally feasible method for studying admissible rules for certain extensions of 
$\text{RM}$.  
See \cite{GM16} and \cite{CFMP} and the references therein for background. % and possible general  methodologies.
Compelling evidence for the power %success
 of the duality-based approach 
in \cite{CFMP}  was given in \cite{CPsug}*{Section~8}. 
There are significant differences between   $\text{RM}$
and $\text{RM}^{\mathbf{t}}$  as regards 
structural completeness properties, with strong assertions being available for the latter. 
This led us  in \cite{CPsug} to focus  on the variety~$\SA$ 
and specifically on its finitely generated quasivarieties.
The  
generators for these quasivarieties are the algebras $\Zed_k$
for $k \geq 1$, where 
$\Zed_k$ has a lattice reduct which is a chain with $\left[
\frac{k+1}{2}\right]$ elements.  Analogously, the finitely
generated quasivarieties of~$\SM$ have generators 
$\W_k$, where $\W_k$ has $\SA$-reduct~$\Zed_k$ and has
$\mathbf t$ interpreted as~$0$ if~$k$ is odd and as~$1$ if~$k$ is even.  
(Details of the definitions are recalled in Section~\ref{sec:prelim}.)
A Sugihara algebra or monoid  is \defn{odd}
if its  involution $\neg$   has a (necessarily unique)  fixed point, and \defn{even} otherwise.  
Thus $\Zed_k$ and $\W_k$ are odd if and only if~$k$ is odd.
The general structure theory of Sugihara monoids, and of their 
relational models, 
 is most  highly developed in the odd 
case;  see in particular
 \cites{FG17,GR12,GR15, OR07}.  

In~\cite{CPsug} we presented 
  dualities  (specifically natural dualities which have the property of being strong) 
for all  the  quasivarieties $\SA_k := \ISP(
\Zed_k)$ of $\SA$.
 Our representation theorems took  
different 
forms in the odd and even cases, with \cite{CPsug}*{Theorem~6.4}
(even case) looking more complicated than \cite{CPsug}*{Theorem~4.3}  (odd case).  
 At the heart of this dichotomy is the fact that
the  quasivarieties $ \ISP(\Zed_{2m-1})$  are varieties whereas  
 $\ISP(\Zed_{2m}) \subsetneq  \HSP(\Zed_{2m})$; 
 the variety generated
by  $\Zed_{2m}$ is $\ISP(\Zed_{2m}, \Zed_{2m-1})$. 
  (We adopt the customary notation for class operators.) 
 Details of these claims, and the analogues for Sugihara monoids,
with $
\Zed_k$ replaced by $
\W_k$,   are given in Section~\ref{sec:prelim}.   
These  observations suggest that we 
should switch  attention from 
quasivarieties to the varieties  they generate.  The machinery 
of multisorted dualities makes this possible, and brings other benefits too.

%%%%%%%%%%%%%%%%%%%%%%%%%%%%%%%%%%%%%%%%%%%

%%%%%%%%%%%%%%%%%%%%%%%%%%%%%%%%%%%%%%%%%
We shall assume that the reader is familiar with
the basic theory of natural dualities, as summarised in black box fashion in \cite{CPsug}*{Section~3}.
 Section~\ref{sec:multisorted} below provides a brief formal introduction
to the use of multisorted dual categories, 
with the theory tailored to our needs.
 We are mindful that 
multisorted duality  theory is less  well known than it deserves to be.  Accordingly,  by way of salesmanship,
we   include a short appendix to our paper to set in context 
 our main results in Sections~\ref{sec:alg-multi} and~\ref{sec:mon-multi}. The appendix, aimed at   
 those new to multisorted dualities,  illustrates  the 
key ideas using  two simple
 examples: Kleene algebras and 
Kleene lattices and shows  how a multisorted
perspective  facilitates making connections between 
natural dualities and  
alternative, Priestley-style, representations.

Once Section~\ref{sec:multisorted}'s   theoretical foundations are in place, we  
 focus on Sugihara algebras and monoids.  Section~\ref{sec:prelim} recalls facts from \cite{CPsug} on Sugihara algebras and assembles the information we need, likewise, about Sugihara monoids.  
Section~\ref{sec:alg-multi}
begins with Theorem~\ref{thm:OddMulti}, the  translation into a two-sorted form of
\cite{CPsug}*{Theorem~4.3}.  We highlight  the similarities and the differences between what are essentially two formulations of  the same result.  By contrast,  Theorem~\ref{thm:EvenMulti} enters new territory.  It supplies a 
three-sorted strong duality  for 
$\CVA{2m}=\HSP(\Zed_{2m})$, the variety generated 
by $\Zed_{2m}  $.   
This duality is easier to work with than the single-sorted
one with $2m-1$ carriers for the quasivariety $\ISP(\Zed_{2m})$
\cite{CPsug}*{Theorem~6.4}.  
In Section~\ref{sec:mon-multi} we capitalise on our work in the algebras case to exhibit Theorems~\ref{thm:OddMultimon} and \ref{thm:EvenMultimon},  the corresponding duality theorems
for the monoids case.
Very little work is involved.  Moreover, in both the odd and even cases,  
the differences between  the $\SA$ and $\SM$ results are limited and
localised in nature,  so that the effect of including or omitting the 
truth constant from the language feeds through in a transparent 
way  to our representation theorems.

Finally in this introduction 
we  draw attention  to an important limitation on
the applicability of  natural duality theory to Sugihara algebras and monoids. 
Each of the varieties $\SA$ and $\SM$ is generated by a single algebra
 having   the integers as its  lattice
reduct.   There are instances of natural dualities based on
 infinite alter egos, for example those for abelian groups 
(the famous Pontryagin duality) and  Ockham algebras.
However
the  requirements for an infinite generating algebra to have 
  a compatible alter ego are stringent (see the general discussion in
\cite{DHP16}) and 
 the varieties $\SA$ and $\SM$ cannot be brought under the natural duality  umbrella this way.
We argue that, nevertheless,  finitely generated subquasivarieties and subvarieties of a variety which is not finitely generated can carry valuable information, especially for locally finite varieties.  
This was illustrated in
\cite{CPsug}*{Section~8}, concerning admissible rules for $R$-mingle. We note too that
every finitely generated free algebra in $\SA$ or $\SM$
already belongs to some finitely generated subvariety
and can be analysed  there;  
 see \cite{CPfree}.  
 More widely, information 
revealed by applying duality methods  to sub(quasi)varieties of a variety may provide pointers to algebraic features of the variety as a whole. 
  We note the recent study by 
Fussner and Galatos  \cite{FG17} of the algebraic structure of
(odd) Sugihara monoids and of  relational  models for them.  At this stage 
there is little overlap  between our work and theirs in that the novelty of our approach rests on the use of multisorted natural dualities whereas they make only very limited use of natural duality theory and then only in a single-sorted form.
 %(see Section~\ref{sec:why-multi}).  

%%%%%%%%%%%%%%%%%%%%%%%%%%%%%%%%%%%%%%%%%%%%%
%%%%%%%%%%%%%%%%%%%%%%%%%%%%%%%%%%%%%%%%%%%%%

%%%%%%%%%%%%%%%%%%%%%%%%%%%%%%%%%%%%
%%%%%%%%%%%%%%%%%%%%%%%%%%%%%%%%%%%%%%%%%%%%% 

\section{The framework of multisorted natural  dualities}\label{sec:multisorted}

Our objective in using duality theory 
is to be able to study a class of algebras $\CA$ by setting up
a  well-behaved dual category equivalence between $\CA$ and
a category  $\CX$ so that problems about~$\CA$ can  be faithfully translated into problems about~$\CX$ which one 
anticipates  will be more tractable.  
Our approach to the admissible rules problem 
originating in \cite{CFMP} and exploited in \cite{CPsug} relies
crucially on the use of strong dualities.  
As the term is used in natural duality theory \cite{CD98},
a  strong duality for a finitely generated quasivariety $\CA = \ISP(\M)$ sets up a dual equivalence between~$\CA$ and a
category $\CX$ of topological relational structures, generated
as a topological quasivariety by~$\MT$, a compatible `alter ego'
for $\M$ which is injective in~$\CX$  (the technical details 
need not concern us here). 
In \cite{CPsug}, and equally in this paper 
and its successors,
the existence of dualities is guaranteed by general theory and our 
assumptions.  But more is at stake: 
 we seek 
dualities 
which are based on economical, and so amenable, 
alter egos.   We achieved this in \cite{CPsug} by using 
the method  of generalised piggybacking which was
introduced in \cite{DP87} and is based on a theory of multisorted
natural dualities.  
But in \cite{CPsug}  multisortedness was kept covert.  Our dualities there used
dual categories whose objects are traditional single-sorted
structures;  we obtained the alter egos determining the dual categories using a multi-carrier version of the piggyback method; 
strongness can also be engineered \cite{CPsug}*{Theorem~3.3}. 

In this paper 
we extol the virtues of employing multisorted 
duality theory  in its full-blown form. Now the  
dual categories
 have objects which are multisorted topological structures.  
As a byproduct,   this gives us the freedom to seek a duality
for a class $\ISP(\CM)$, where $\CM$ is a finite set of finite algebras over a common language. 
% \textit{Inter alia}, this allows us to encompass finitely generated varieties which are not quasivarieties,
%what  we shall do for $\HSP(\Zed_{2m})= \ISP(\Zed_{2m}, \Zed_{2m-1})$  in 
%Theorem~\ref{thm:EvenMulti}.   
%This new duality is three-sorted.  As compared with its single-sorted
%counterpart for $\ISP(\Zed_{2m})$  the new  duality is appealingly 
%simple.  We also provide 
%a two-sorted duality for
%$\ISP(\Zed_{2m-1})$, equivalent 
%to 
%\cite{CPsug}*{Theorem~4.3}.   
%It is of particular interest  here that the passage from a  quasivariety  $\ISP(\M)$ to the variety $\HSP(\M)$ can lead to a duality  which is simpler to work with.  
%This phenomenon promises to be  
%valuable  for applications, especially if, as seems likely,
%it is not confined to the Sugihara example.  
%(Of course,  free algebras in a quasivariety
%$\ISP(\M)$ exist, and can equally well be calculated in 
%$\HSP(\M)$.)
%

By way of background,  we remark that 
researchers who have driven  
natural duality theory forward over the past twenty years have 
concentrated on the single-sorted case. % for several reasons.  
%Firstly, 
They were aware that extensions to the multisorted case would be possible, and often straightforward, but that 
the heavier notation involved in working in maximum generality could obscure the underlying ideas;  \cite{DHP16}, a recent major review of  single-sorted piggybacking, 
% confined to the single-sorted case, 
is a case in point.  Moreover, the multisorted \emph{theory} has  evolved
on a `need-to-know' basis,   as potential applications have emerged.

The treatment of multisorted piggybacking and  strong dualities in
\cite{CD98}*{Chapter~7}  omits details and proofs.  
It does anyway not quite meet our needs. We shall instead draw on the paper \cite{DT03} by Davey and Talukder.   This includes a clear summary of the basic theory and also the technical result on strongness that we shall require.  (Furthermore, \cite{DT03} studies dualities for finitely generated varieties of Heyting algebras, so the algebras involved there have a connection with 
Sugihara algebras and monoids, though this relationship is not pursued in 
the present paper.)

We do not need the theory in the most general form possible.
The classes of algebras we wish to consider 
will be of the form 
$\CV(\M)= \HSP(\M)$ or
$\CQ(\M)$, where $\M$ is  finite and has a  distributive 
lattice reduct;  here $\CV(\M) = \HSP(\M)$  and $
\CQ(\M)= \ISP(\M)$ are, respectively, the variety and quasivariety 
generated by~$\M$.  On the universal algebra front, these 
assumptions ensure that $\CV(\M)$ is congruence distributive and that J\'onsson's Lemma applies.  This means that 
$\CV(\M)$ will be expressible as $\ISP(\CM)$, 
where $\CM = \{ \M_1, \ldots, \M_N\}$ is finite and each 
$\M_i \in  \mathbb{HS}$.  As regards duality theory,  we will have access to both the Multisorted Piggyback Duality Theorem
(for piggybacking over the variety  $\CCD_u$  of all distributive
lattices) \cite{CD98}*{Theorem 7.2.1}  and the Multisorted NU Strong Duality  Theorem, making use of the fact that we are dealing with lattice-based algebras, \cite{CD98}*{Theorem 7.1.2}.

	So consider  a class $
\CA = \ISP(\CM)$, where $\CM$ is a finite set of finite algebras
over a common language and having reducts in $\CCD_u$. 
We regard $\CA$ as a category, in which the morphisms are all 
homomorphisms. 
 We seek a dual category $\CX$ whose objects are    multisorted topological  structures and whose morphisms are 
maps which preserve the sorts and are continuous
and structure-preserving.  A  background reference for
multisorted structures is \cite{KK}.
We want to construct $\CX$ in such a way 
that there are functors $\fnt D \colon \CA \to \CX$ 
and $\fnt E \colon \CX \to \CA$ setting up a dual equivalence.  
This needs to  be done in a very specific way, so that  
$\fnt D $ and $\fnt E$ are given by appropriately defined 
hom-functors.

%%%%%%%%%%%%%%%%%%%%%%%%%%

We now explain what constitutes an admissible 
\defn{alter ego}  $\CMT$ for $\CM$, how to construct the dual category  $\CX$ of
 multisorted structures generated by $\CMT$ 
and  set up   a   dual adjunction between 
$\CA$ and $\CX$. 
Our alter ego for  $\CM$  takes  the form 
$
\CMT = (M_1 \du \cdots  \du M_N;  G, H, K, R, \Tp)$.
Here $R$ is  a  set of relations each of which is a subalgebra 
of some $\M_i \times \M_j$, where $i,j \in \{1,\dots,N\}$. 
The sets $G$ and~$H$ consist, respectively, of homomorphisms 
and (non-total) partial homomorphisms, each from 
 some $\M_i$ into some $\M_j$.  The elements of~$K$ 
are the one-element subalgebras of the members of 
$\CM$.  
The alter ego $\CMT$ carries  the disjoint union topology derived
 from the discrete topology on the individual sorts  $M_i$. 
(These assumptions parallel those in the single-sorted case; see
\cite{CPsug}*{Section~3}.)

We form multisorted topological 
$\CMT$-structures~${\X = \X_1 \du \cdots \du \X_N} $ 
where each of the sorts $\X_i$ is a Boolean  space,~$\X$ 
is equipped with the disjoint union topology and
regarded as a structure, $\X$  is of the same type as 
 $\CMT$.  
Thus $\X$ is equipped with a set $R^{\X}$ of relations $r^{\X}$; 
if $r \subseteq M_i \times M_j$, then 
${r^{\X} \subseteq X_i \times  X_j}$;  and similar  statements 
apply to $G^{\X} $, $H^{\X}$ and $K^{\X}$.  
 Given $\CMT$-structures $\X $ and $\Y$, a morphism 
$\phi \colon \X \to \Y$ is  defined to be a continuous map
 \emph{preserving the sorts}, so that $\phi(X_i) \subseteq Y_i$, 
and  $\phi$ preserves the structure.
The terms isomorphism, embedding, etc., 
are then defined appropriately. As in the single-sorted case,
care needs to be taken with embeddings when $H \neq \emptyset$; see \cite{CPsug}*{Section~3}.

We define our dual category $\CX$ to have as objects those
 $\CMT$-structures $\X$ which belong to  a  class of  
topological structures which we shall denote by $\IScP(\CMT)$.  
Specifically, $\CX$ consists of isomorphic copies of closed 
substructures of powers of $\CMT$. 
Here powers are formed  `by sorts': given a non-empty set~$S$, 
the underlying set of $\CMT^S$ is the union of disjoint copies
 of~$M^S$, for $\M \in \CM$, equipped with the disjoint union
topology obtained when each $M^S$ is given  the product topology.
The  structure given  by $R$,  $ G$, $H$ and $K$  is lifted pointwise to 
 substructures of such powers.
The superscript $^+$ indicates that the  empty structure is 
included in $\CX$.

We set up   hom-functors 
$\D \colon\CA \to \CX$ and $\E\colon \CX \to\CA$ 
using 
$\CM$
and its alter ego $\CMT$: 
\begin{alignat*}{2} 
                \D (\A)&= \CA(\A,\M_1) \du \cdots \du \CA(\A,\M_N), \qquad &
                \D (f)&  = - \circ f; \\
    \E (\X)& = \CX(\X,\CMT),  &
             \E (\phi)  &= - \circ \phi .  
\end{alignat*}
Here the disjoint union 
$\CA(\A,\M_1) \du \cdots \du \CA(\A,\M_N)$ is  a 
(necessarily closed) substructure of $M_1^A\du \cdots \du M_N^A$,
 and so a member of $\IScP(\CMT)$. 
 We recall from above that  $\CX(\X,\CMT)$, as a set, 
is the collection of continuous structure-preserving maps
 $\phi\colon \X \to\CMT$ which are such that 
$\phi(X_i)\subseteq M_i $ for $1 \leq i \leq N$.
This set acquires the structure of a member of 
$\CA$ by virtue of viewing it as a subalgebra of the power 
$\M_1^{X_1} \times  \dots \times \M_N^{X_N}$. 
The well-definedness of the functors $\D$ and $\E$ is of central 
importance to our enterprise. 
It hinges on the algebraicity assumptions we have made 
on $G, H, K, R$ 
and the requirement that each $M_i$ is 
finite and carries the discrete topology.
Moreover, 
 $\D$ and $\E$ set up  a  dual adjunction, $(\D,\E, e,\varepsilon)$ 
in which the unit and counit maps are evaluation maps,  
and these evaluations  are embeddings.

We say $\CMT$
 \defn{yields a multisorted duality} if, for each 
$\A \in \CA$,  the evaluation map $e_{\A}\colon
\A \to \E\D(\A)$
is an isomorphism.  
 The duality is \defn{full} 
if, for each $\X \in\CX$, the evaluation map 
$\varepsilon_{\X} \colon \X \to \D\E(\X)$ is an isomorphism,  and the categories~$\CA$ and~$\CX$ are dually equivalent. 
Thus a duality provides a concrete 
representation $\E\D(\A)$ of  $\A \in \CA$.
If in addition the duality is full,  every 
$\X \in \CX$ arises, up to isomorphism,  as a topological structure 
$\D(\A)$, for some $\A \in \CA$.  
We record  an important fact, true for  any multisorted duality, 
and adding weight to the duality's claim to be called `natural'.  
In $\CA = \ISP(\CM)$, the free algebra $\F_{\CA}(s)$  on~$s$ 
generators is isomorphic to $\E(\CMT^S)$
 \cite{CD98}*{Lemma~2.2.1 and Section~7.1}.

%hp1902  shortened 
%For many applications we require a duality which is full.  
In practice, fullness of a duality  is normally obtained at 
second hand by showing that the duality is  \defn{strong}.  
In certain applications---and this was crucial 
 in the TSM method for testing admissibility \cites{CFMP, CPsug}---consequences of strongness are required,
whereby
each of the functors setting up 
a strong duality 
converts an embedding to a surjection and a surjection to an embedding. 
Strongness of a single-sorted duality can be defined in several equivalent ways and the same is to be expected of  a multisorted
duality.   
In accordance with our policy of a black-box treatment of duality theory,  we shall suppress the details as far as possible.
%  In addition, 
%\cite{DT03}*{Theorem~4.1} shows 
% that a multisorted strong duality is full.  
%

The multisorted piggyback duality theorem originated in \cite{DP87}*{Theorem~2.2}    and thereafter  various specialisations
have been well documented, as for example in  
\cite{CD98}*{Theorem~7.2.1},
Most previous applications in the literature piggyback over 
$\cat{D}$, bounded distributive lattices with $0,1$ included in the language.  However,
as in \cite{CPsug}, we shall, except briefly in the appendix,  
piggyback over $\cat{D}_u$, the variety of all distributive lattices.
The upgrade to a
strong duality was not discussed in \cite{DP87}.  For this
 see for example
 \cite{CD98}*{Theorem~7.1.2} (the Multisorted 
NU Strong Duality 
Theorem and \cite{DT03}*{Theorem~2.1}.

We now state the multisorted duality theorem that  we shall apply to quasivarieties and varieties of Sugihara algebras and monoids.  
The reason for including the condition on subalgebras being 
subdirectly irreducible is that it ensures that partial operations
of arity $>1$ are not needed; cf.~\cite{CD98}*{Section~7.1}.  
A discussion of the role of the set~$K$ of one-element subalgebras of members of $\CM$, and the need, for the application of the Multisorted NU Strong Duality Theorem, to include these in $\CMT$,  can be found in
\cite{DT03}*{Lemma~4.5}.

\begin{thm}  
[{\rm Multisorted Piggyback Strong Duality Theorem, for distributive-lattice-based algebras}]\label{thm:multpigstrong}
Let $\CA = \ISP(\CM)$, where $\CM = \{ \M_1, \ldots,\M_N\}$ 
is a finite set of 
pairwise disjoint finite algebras in~$\CA$ and  assume  that there is a  forgetful functor $\fnt U \colon \CA \to  \cat{D}_u$  such 
that $\fnt U(\A)$ is a reduct of~$\A$ for each $\A \in \CA$.  Assume in addition  that for each~$\M \in \CM$ every non-trivial subalgebra 
of $\M$ is subdirectly irreducible. 

For each $\M\in \CM $  let $\Omega _{\M}$ be a {\rm(}possibly
empty{\rm)} subset of $\cat{D}_u(\fnt{U}(\M), \two )$, where~$\two$ denotes the $2$-element lattice in $\cat{D}_u$.

Let $\CMT$
 be the discretely topologised relational 
structure
with universe  $M_0:= M_1 \du \cdots \du M_N$ and equipped with
\begin{enumerate}  
\item[{\rm (i)}] $G$,   a subset of
$\bigcup \{\, \CA(\M,\M')\mid \M, \M'\in \CM \,\}$ 
satisfying the following separation condition: 
 for all $\M\in \CM$, given 
$a,b\in M$  with $a \ne b$, there exist
$\M' \in \CM$ and 
 $u\in \CA(\M , \M')\cap G$ 
 and 
$\w' \in \Omega _{\M'}$  such that 
$
\w' (u(a))  \ne \w'(u(b))  
$;
\item[{\rm (ii)}] 
  $R$, the collection of maximal $\CA$-subalgebras
 of sublattices of the form
  $$
(\w,\w')^{-1}(\leq) :=
\{\, (a,b) \in M\times  M' \mid \w(a) \leq \w' (b)\, \} , 
$$
for which  $\w\in \Omega _{\M}$, $\w' \in \Omega _{\M'}$
{\rm(}where  $\M, \M'$ range over $\CM${\rm)}. 
\end{enumerate}
Then $\CMT $ yields a duality on $\CA$.

 If the structure of the alter ego is augmented so as to include                                 
\begin{enumerate}
\item[{\rm (iii)}] $H$, the  set 
of all homomorphisms~$h$  whose domain is a subalgebra 
of some $\M$ and whose image is a subalgebra of some 
$\M$, where $\M, \M' \in \CM$;
\item[{\rm (iv)}]  $K$, the set of elements which
are universes of  one-element subalgebras in some member of 
$\CM$, viewed as nullary operations of~$\CMT$.
\end{enumerate}
then $\CMT$ yields a strong and hence full duality on $\CA$.  
\end{thm}

Some  comments are in order here.  
  Recall that 
Theorem~\ref{thm:multpigstrong} is derived by pasting together sufficient conditions for a reasonably economical duality
(by multisorted piggybacking)  and for an upgrade to a strong duality  (Multisorted  NU  Strong Duality Theorem).  
The theorem ensures that, when conditions (i)--(iv) are satisfied
  a strong duality is available.  But 
we would be likely to want 
to massage the alter ego so delivered in order to arrive at
 an  alter  ego which is easier to work with but which still strongly dualises~$
\CA$.
  A  result that facilitates the streamlining 
of an alter ego  in the single-sorted case is Clark and Davey's 
$\MT$-Shift Strong Duality Lemma,  of which we  presented a simplified for  in~\cite{CPsug}*{Lemma 3.4}.  We now present  the multisorted
version we shall need in this paper.

\begin{lem}[Multisorted $\CMT$-Shift  Strong Duality Lemma]   
\label{lem:multi-shift}
 Assume that  $\CA =\ISP(\CM)$ is a  finitely generated  quasivariety and that it is strongly dualised by   
$\CMT = (M_0; G,H,K,R,\Tp)$.
Then an alter ego $\CMT' = (M_0; G', H',K', R',\Tp)$ will also yield a strong duality if any of the following applies:
 \begin{newlist}
\item[{\rm (a)}]  $ R' \supset R$;
\item[{\rm (b)}]  $H'$ is obtained by deleting from  $H$ any element
 expressible  as a composition of the elements that remain;
\item[{\rm (c)}]    
  $\CMT'$ yields a duality on $\CA$ and is obtained from $\CMT$ by deleting 
members of $R\cup K$;  
\item[{\rm (d)}]  $\CMT'$ yields a duality on $\CA$ and is obtained by deleting from $H$ any element  which has a proper extension belonging to  $G\cup H$.
\end{newlist}
\end{lem} 

\begin{proof}  
As in the single-sorted case the proofs of the duality claims in (a) and (b) are 
 bookkeeping exercises with evaluation maps; 
cf.~\cite{CDquest}*{Lemmas~1.4 and~3.1}.  The part of the  proof of \cite{CDquest}*{Lemma~3.1} that deals with strongness
relies on the fact that a closed substructure of a non-zero power of $\CMT$ 
in the dual category is hom-closed.   The multisorted 
version of this assertion is stated and proved in 
\cite{DT03}*{Lemma~4.3}.  The proofs of (b)--(d) then proceed as in the single-sorted case.   
(Here (b) is a restricted form of hom-entailment, sufficient for our
purposes; 
see \cite{CD98}*{Section~9.4}.)
\end{proof}

The Shift Lemma, items (b) and (d),  promises to be valuable for 
eliminating unnecessary  maps from the set stipulated in
Theorem~\ref{thm:multpigstrong}(iv).   Moreover, in our applications, 
  the piggyback relations  in (ii)   will turn out to  be graphs of partial  endomorphisms of sorts or of partial homomorphisms
between sorts, or converses of such graphs.  
But, assuming 
we want our duality to be strong,  we do need to include partial 
operations in the alter ego rather than their graphs.  
%lc2301: I am not sure if to include this last sentence. We haven't talked about hom-entailment and I thought we weren't going to mention it. If the idea is to show that we are aware of it and its role in strong dualities, it is fine, but I would add one reference, perhaps to the book where hom-entailment is defined.
%hp2301  taken out, but I *had* mentioned hom-entailment in shift lemma proof.
% The point here is that the technical notion underpinning the 
%Shift Lemma is 
%hom-entailment.  

%We make some further comments on these matters in relation
%to our dualities at the end of 
%Section~\ref{sec:mon-multi}, once all the dualities are in place.
%%%%%%%%%%%%
%hp1702  new short bit, in lieu of having the chat at the end of Sec 5, 
%which we agreed to take out from there

There is no \textit{a priori} reason to expect that there will be a  close 
relationship between the
set-up envisaged in the first part of  Theorem~\ref{thm:multpigstrong}, in which we merely seek a  duality, and the additional structure which is added to ensure strongness.  
However we shall see that,  in the Sugihara context, 
the operations in~(iii) are tightly connected to the relations in~(ii), for the carrier maps we select.   
 Indeed, 
this is precisely what 
allows us to present streamlined alter egos, in which piggyback 
relations \textit{per se} do not feature.  In certain applications---the
description of finitely generated free algebras for Sugihara algebras and monoids being a prime example
(see \cite{CPfree})---we do not need strongness, or even fullness.  But knowing this does not lead to any worthwhile simplifications.  
%%%%%%%%   

%%%%%%%%%%%%%%%%%%%%%%%%%%%%%%%%%%%%%%%%%%%%%%
%%%%%%%%%%%%%%%%%%%%%%%%%%%%%%%%%%%%%%%%%%%%%%
\section{Sugihara  algebras and monoids,  odd and even cases}  \label{sec:prelim}

Given the freestanding treatment of the class $\SA$ of Sugihara 
algebras available in \cite{CPsug}
our aim as far as possible is to draw on the Sugihara algebra results to arrive at corresponding results for monoids, to avoid starting afresh.  Therefore
 we shall regard the class $\SM$ of Sugihara monoids as being obtained by enriching the language of $\SA$, instead of going in the opposite direction.

We recall that the variety $\SA$ of \defn{Sugihara algebras} is generated by the
algebra~$\Zed$ whose underlying lattice is the chain of integers.  The negation and implication are given, respectively,  by
\[
\neg a = -a  \quad \text{and} \quad 
a\to b = \begin{cases}
(-a) \vee b & \mbox{if }a\leq b,\\
(-a) \wedge b & \mbox{otherwise.}
\end{cases}
\]
%hp1802  rejigged because of horrid linebreak problems.
A finite Sugihara algebra is subdirectly irreducible 
if and only if it is (isomorphic to) some subalgebra $\Zed_k$  ($k \geq 2$).    This result  dates back to Blok and Dzobiak 
\cite{BD86};  see also  \cite{CPsug}*{Section~2}.  
Here the universe of    
$\Zed_{2n+1}$ is  $\{\, a \in \Zed \mid
  -n \leq a \leq n \,\}$ ($n \geq 0$) and $\Zed_{2n} = 
\Zed_{2n+1} \setminus \{0\}$ ($n \geq 1$).  
%%%%%%%%%%%%%%%%%%%

The variety 
$\SM$ of  \defn{Sugihara monoids} consists of  algebras 
$\A = (A; \land, \lor, \to, \neg, \mathbf t)$ such that the
reduct
$(A; \land, \lor, \to, \neg)$ belongs  to $\SA$ and~$\mathbf t$ is a constant which  is 
an identity for 
the derived fusion operation given by  $a \cdot b =
\neg (a \to \neg b)$.   It is useful to denote the derived constant $\neg \mathbf t$ 
by $\mathbf f$.
When~$k$ is odd. the Sugihara monoid~$\W_k$  
associated with the Sugihara algebra $\Zed_k$ has  
 the constants $\mathbf t$ and~$\mathbf  f$  interpreted 
as~$0$.  When $k$ is even,  $\mathbf t$ is interpreted as $1$ and $ \mathbf f$ as $-1$.
It has long been known  (see \cite{Du70} and also \cite{AB75}*{pp.~422--423}) that the finite subdirectly 
irreducible  members of $\SM$ are the algebras~$\W_k$
for $k > 1$.  We let $\SM_k:= \ISP(\W_k)$.

We next assemble information about finitely generated 
quasivarieties and varieties in~$\SA$ and~$\SM$, drawing on some fundamental results from universal algebra.  
Suppose we have a variety $\CA:= \CV(\M)$,
 where~$\M$
is  some finite algebra~$\M$.
Let $\class{S}:=\si(\CA)$ be the set 
of subdirectly irreducible algebras in~$\CA$, up to isomorphism.  Then Birkhoff's Subdirect Product Theorem tells us that $\class{S}$ is finite and that  
$\HSP(\M) =\ISP(\class{S})$.
  We shall be dealing with  the situation in which 
$\CA $ is congruence-distributive. A corollary of 
  J\'onsson's Lemma 
asserts that  in such a variety, every subdirectly irreducible algebra 
is a homomorphic image of a subalgebra of~$\M$.
In addition,
the lattice of subvarieties of~$\CA$ is a finite
distributive lattice.  
It follows that
of subvarieties of~$\CA$
 is isomorphic to the lattice of down-sets 
of $\class{S}$,   where $\A \leq \B $ in $\class{S}$ if and only if $\A \in \mathbb{HS}(\B)$    (see for example \cite{BD79}).

Take $m \geq 2$.
We  note  that $\Zed_{2m-1}$ is a homomorphic image of~$\Zed_{2m}$ via the homomorphism $u$ given by 
\[
u(a) =  \begin{cases}
                   a-1  &\text{if  }  0 < a \leq m,\\
                    a+1  &\text{if  }  -m\leq a < 0. 
                \end{cases}
\]

Proposition~\ref{prop:varsetc} leads to the diagram in 
\cite{CPsug}*{Figure~1}  for the lattice of finitely generated
subquasivarieties of~$\SA$.  We omitted the justification there since we did not explicitly need the result.  In this paper we do need  it. 

  In \cite{CPsug} we considered the chains
of quasivarieties $\{\SA_{2n+1}\}_{n\geq 1}$
(odd case) and $\{\SA_{2n}\}_{n \geq 1}$ (even case) and presented dualities for $
\SA_{2n+1}$ and $\SA_{2n}$, for general~$n$.  In 
Proposition~\ref{prop:varsetc} we have changed labels from~$n$
to $m$.  Part (ii) explains why.  In order to provide a multisorted
duality for % the variety generated by
$\CVA{2n}$ we would expect to draw on ingredients from the 
dualities for $\Zed_{2n}$ and $\Zed_{2n-1}$ at the  same time
and it is convenient subsequently to label the odd case 
quasivarieties as $\{ \SA_{2m-1}\}_{m\geq 2}$ and the even case
quasivarieties as $
\{\SA_{2m} \}_{m\geq 2}$.  (The class $\SA_1$ is the trivial class and $\SA_2$ is term-equivalent to the variety of Boolean algebras, and we have nothing to say about either class.)

\begin{prop}[Sugihara algebras: quasivarieties and varieties]  \label{prop:varsetc}\

\begin{enumerate}

\item[{\rm (i)}] $
\ISP(\Zed_{2m-1}) 
= \HSP(\Zed_{2m-1})$, and so is a variety.

\item[{\rm (ii)}] 
$\ISP(\Zed_{2m}) \subsetneq  \HSP(\Zed_{2m})$ and 
 $\HSP(\Zed_{2m}) = \ISP (\Zed_{2m}, \Zed_{2m-1} )$.   

\item[{\rm (iii)}]  
$\HSP(\Zed_1) \subset \HSP(\Zed_2) \subset \HSP(\Zed_3)\subset \cdots \subset \SA$. 
\end{enumerate}
\end{prop} 

\begin{proof}  
To apply  J\'onsson's Lemma to the variety $\CV(\Zed_k)$ we need to investigate
$\mathbb{HS}(\Zed_k)$.  
By \cite{CPsug}*{Proposition~2.1}, any given subalgebra of~$\Zed_k$ is isomorphic to some $\Zed_r$, where $r\leq k$ 
and $r$ is even if $k$ is even.   Now consider 
a homomorphism~$h$ with domain  $\Zed_r$.  Its image 
$\img h$ is isomorphic to  a quotient  $\Zed_r/\theta$, 
where~$\theta$ is a congruence on $\Zed_r$.    According to \cite{CPsug}*{Proposition~2.4}, there exists $p \leq \left[\frac{r+1}{2}\right]$ such that $a \, \theta \, b$ if and only if 
$a = b$ or $|a|, |b| \leq p$.  Thus $\Zed_r/\theta
\cong \Zed_{2(s-p)+1}$. 
We now separate the odd and even cases.

Assume $k =2m-1$.  In this  case 
$s \leq p$.   It follows that $\Zed_{2(s-p)+1}\in \ISP(\Zed_k)$.
This implies that every subdirectly irreducible algebra in
$\HSP(\Zed_{2m-1})$ belongs to $\ISP(\Zed_{2m-1})$ and hence
that $\HSP(\Zed_{2m-1}) \subseteq \ISP(\Zed_{2m-1})$.  
Hence  the two classes coincide.  We have proved (i).

Assume $k=2m$.   Any member of $\mathbb{HS}(\Zed_{2m})$ which is already in $\mathbb{S}(\Zed_{2m})$ certainly belongs to
$\ISP(\Zed_{2m}, \Zed_{2m-1})$.  Now suppose that we have a non-trivial congruence $\theta$ on some subalgebra $\Zed_r$, with $r$ even 
and $r \leq 2m$.  From above,  $\Zed_r/\theta \cong 
\Zed_{2(s-p)+1}$. Since $s =r/2 \leq m$ and $p \geq  1$, the quotient is isomorphic to a subalgebra of 
$\Zed_{2m-1}$ and so in $\ISP(\Zed_{2m}, \Zed_{2m-1})$. Clearly
$\ISP(\Zed_{2m})$ is not a variety for $m > 1$ since it does not 
contain $\img u = \Zed_{2m-1}$.
We deduce that  (i) holds.

Finally (iii) is now immediate from (i) and (ii).
\end{proof}

The following proposition parallels the content of 
Proposition~\ref{prop:varsetc} and can be proved in the same way.  Note that $\W_{2m}$ has $\W_{2m-1}$ as an $\SM$-morphic image:  the surjective $\SA$-morphism $u \colon \Zed_{2m} \to 
\Zed_{2m-1}$ defined above is an $\SM$-morphism.

\begin{prop}[Sugihara monoids: quasivarieties and varieties]  \label{prop:varsetcmon}\

\begin{enumerate}

\item[{\rm (i)}] $
%\SM_{2m-1}
 \ISP(\W_{2m-1})
= \HSP(\W_{2m-1})$, and so is a variety.

\item[{\rm (ii)}] 
$\ISP(\W_{2m}) \subsetneq  \HSP(\W_{2m})$ and
 $\HSP(\W_{2m}) = \ISP (\W_{2m}, \W_{2m-1} )$.   

\item[{\rm (iii)}]  The inclusion orderings of 
the subquasivarieties $\{ \ISP(\W_k)\}_{k\geq 1}$ of $\SM$
and 
 of the subvarieties 
$\{ \HSP(\W_k)\}_{k\geq 1}$ of $\SM$ are as shown in Figure~\upshape{\ref{fig:MonoidVars}}.
\end{enumerate}
\end{prop}

\begin{proof}  (i) and (ii) 
  can be proved in the same way as the corresponding claims
for Sugihara algebras (see also for example \cite{OR07}*{Section~3}). 
% We may also imitate the proof in 
%Proposition~\ref{prop:varsetc}  to deduce that the finitely generated
%subquasivarieties of~$\SM$ are ordered in exactly the same way as those of~$\SA$ are.  
It is trivial that $\W_{k}\in\ope{S}(\W_{k+2})$ for each $k$ and that $\W_{2m-1}\in\ope{H}(\W_{2m})$.
It 
 is not true that $\W_{2m}$ belongs to $\HSP(\W_{2k+1})$ (or to $\ISP(\W_{2k+1})$) for any $k$.  If it did, 
$\W_{2n}$ would have to be a homomorphic image of a subalgebra of $\W_{2k+1}$.  But this is  impossible since the constants~$\mathbf t$ and~$\mathbf f$ are distinct in $\W_{2m}$ but coincide 
in any member of $\HSP(\W_{2m+1})$. 
It is also not true that $\W_k$ belongs to any $\HSP(\W_\ell)$ when $\ell<k$.
\end{proof}

%
%%%%%%%%%%%%%%%%%%%%%%

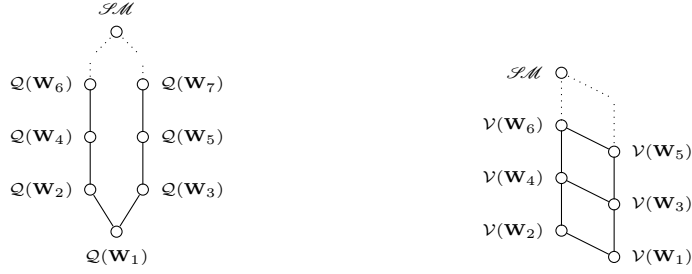
\begin{figure}
\begin{center}
\small{ 
	\begin{tikzpicture}[scale=.7]
	%odd
		\node [label=right:{$\scriptstyle \CQM{7}$}] (S7) at (1,4) {};
		\node [label=right:{$\scriptstyle \CQM{5}$}] (S5) at (1,3) {};
		\node [label=right:{$\scriptstyle \CQM{3}$}] (S3) at (1,2) {};
		\node [label=below:{$\scriptstyle \CQM{1}$}] (S1) at (0.5,1.2) {};

   		\draw [shorten <=-1.414pt, shorten >=-1.414pt] (S7) -- (S5);
		\draw [shorten <=-1.414pt, shorten >=-1.414pt] (S5) -- (S3);
	
		\draw (S1) circle [radius=3pt];
		\draw (S3) circle [radius=3pt];
		\draw (S5) circle [radius=3pt];
		\draw (S7) circle [radius=3pt];
	%% even
		\node [label=left:{$\scriptstyle \CQM{6}$}] (S6) at (0,4) {};
		\node [label=left:{$\scriptstyle \CQM{4}$}] (S4) at (0,3) {};
		\node [label=left:{$\scriptstyle \CQM{2}$}] (S2) at (0,2) {};

   		\draw [shorten <=-1.414pt, shorten >=-1.414pt] (S2) -- (S4);
		\draw [shorten <=-1.414pt, shorten >=-1.414pt] (S6) -- (S4);

		\draw (S2) circle [radius=3pt];
		\draw (S4) circle [radius=3pt];
		\draw (S6) circle [radius=3pt];
		
%%cross connections
%		\draw [shorten <=-1.414pt, shorten >=-1.414pt] (S2) -- (S3);
		\draw [shorten <=-1.414pt, shorten >=-1.414pt] (S1) -- (S2);
		\draw [shorten <=-1.414pt, shorten >=-1.414pt] (S1) -- (S3);
%		\draw [shorten <=-1.414pt, shorten >=-1.414pt] (S4) -- (S5);
%		\draw [shorten <=-1.414pt, shorten >=-1.414pt] (S6) -- (S7);
%% Top part
		\node [label=above:{$\scriptstyle \SM$}] (S) at (0.5,5) {};
		\node  (S-) at (0,4.5) {};
		\node  (S+) at (1,4.5) {};
		\draw (S) circle [radius=3pt];
   		\draw [shorten <=-1.414pt, shorten >=-1.414pt,dotted] (S7) -- (S+);
  		\draw [shorten <=-1.414pt, shorten >=-1.414pt,dotted] (S-) -- (S);
  		\draw [shorten <=-1.414pt, shorten >=-1.414pt,dotted] (S+) -- (S);
 		\draw [shorten <=-1.414pt, shorten >=-1.414pt,dotted] (S-) -- (S6);
		
	\end{tikzpicture}
\hspace*{3cm}
\begin{tikzpicture}[scale=.7]
	%even
		\node [label=left:{$\scriptstyle \CVM{6}$}] (S6) at (0,4) {};
		\node [label=left:{$\scriptstyle \CVM{4}$}] (S4) at (0,3) {};
		\node [label=left:{$\scriptstyle \CVM{2}$}] (S2) at (0,2) {};

   		\draw [shorten <=-1.414pt, shorten >=-1.414pt] (S6) -- (S4);
		\draw [shorten <=-1.414pt, shorten >=-1.414pt] (S4) -- (S2);
	
		\draw (S2) circle [radius=3pt];
		\draw (S4) circle [radius=3pt];
		\draw (S6) circle [radius=3pt];
	%% odd
		\node [label=right:{$\scriptstyle \CVM{5}$}] (S5) at (1,3.5) {};
		\node [label=right:{$\scriptstyle \CVM{3}$}] (S3) at (1,2.5) {};
		\node [label=right:{$\scriptstyle \CVM{1}$}] (S1) at (1,1.5) {};

   		\draw [shorten <=-1.414pt, shorten >=-1.414pt] (S1) -- (S3);
		\draw [shorten <=-1.414pt, shorten >=-1.414pt] (S5) -- (S3);

		\draw (S1) circle [radius=3pt];
		\draw (S3) circle [radius=3pt];
		\draw (S5) circle [radius=3pt];
		
%%cross connections
		\draw [shorten <=-1.414pt, shorten >=-1.414pt] (S1) -- (S2);
		\draw [shorten <=-1.414pt, shorten >=-1.414pt] (S3) -- (S4);
		\draw [shorten <=-1.414pt, shorten >=-1.414pt] (S5) -- (S6);
%% Top part
		\node [label=left:{$\scriptstyle \SM$}] (S) at (0,5) {};
		\node  (S-) at (1,4.5) {};
		\draw (S) circle [radius=3pt];
   		\draw [shorten <=-1.414pt, shorten >=-1.414pt,dotted] (S6) -- (S);
  		\draw [shorten <=-1.414pt, shorten >=-1.414pt,dotted] (S-) -- (S);
 		\draw [shorten <=-1.414pt, shorten >=-1.414pt,dotted] (S-) -- (S5);
		
	\end{tikzpicture}
}
\end{center}
\caption{Sugihara monoids: quasivarieties and varieties \label{fig:MonoidVars}}
\end{figure}

The cornerstone for our development of amenable 
single-sorted 
dualities for the quasivarieties $\SA_k$, for
$k$ both odd and even, was our analysis of   the partial endomorphisms of the algebras $\Zed_k$. The results in
\cite{CPsug}*{Section~2} are equally important in the multisorted
case and we shall want their analogues for Sugihara monoids too. 

 We  recall for reference the definitions of the maps from 
$\Zed_{2m-1}$ to itself that
\cite{CPsug}*{Proposition~2.9} shows constitute a generating set for
$\PEZ{2m-1}$:
\allowdisplaybreaks
\begin{alignat*}{2}
\shortintertext{partial 	 endomorphisms:} 
& f_0 
\colon \Zed_{2m-1} \setminus \{ 0\}  \to \Zed_{2m-1} , 
%\\    
\quad &&    f_0 
(a) = a \mbox{ for  $a\ne 0 $;}
\\[1ex]
& f_1  
\colon \Zed_{2m-1}\setminus \{1,-1\} \to \Zed_{2m-1},
\qquad   &&   f_1 
(a) = a \mbox{ for $a \ne \pm 1$};\\[1.5ex]
 & \text{and for } 1< i < m,\\[-.5ex]
&f_i
\colon \Zed_{2m-1} \setminus\{i,-i\} \to \Zed_{2m-1} , 
%\\&
\qquad &&f_i(a) = 
\begin{cases}
i& \mbox{if } a = i-1,\\
-i& \mbox{if } a = -(i-1),\\
a & \mbox{otherwise};
\end{cases} 
\\[1.5ex]
\shortintertext{endomorphism:}
& g\colon \Zed_{2m-1} \to \Zed_{2m-1} , 
\qquad && g(a) = \begin{cases}
a-1 & \mbox{if } a>0, \\
a+1& \mbox{if } a<0,\\
0& \mbox{if } x=0.
\end{cases}
\end{alignat*}
%%%%%%%%%

We remark that, of course, any morphism between Sugihara monoids necessarily includes the constants~$\mathbf t$ and 
$\mathbf f$  in its domain and image.  This ensures 
that the family of partial endomorphisms of a Sugihara monoid forms a monoid under composition.  (Cf.~the comments 
in \cite{CPsug}*{} for the algebra case.)

We shall adopt the convention that the same symbol is used for  two morphisms, one of Sugihara algebras and the other of 
Sugihara monoids, when they have the same  underlying set map. 
However we shall distinguish sets of (partial) endomorphisms
and, later, (partial)  homomorphisms of Sugihara monoids from their Sugihara algebra
counterparts by including a superscript $^{\mathbf t}$.  This usage 
accords with that used for R-mingle logics.  For partial endomorphisms of Sugihara monoids we have the following analogue of \cite{CPsug}*{Propositions~2.8 and~2.9}.

%%%%%%%%%%%%%%%%%%%%%%%  

\begin{prop}[partial endomorphism of Sugihara monoids, odd case]\label{lem:monpe:odd} \

\begin{enumerate}
\item[{\rm (i)}]
The endomorphisms of $\W_{2m-1}$ are the same as those of 
$\Zed_{2m-1}$.
\item[{\rm (ii)}]
 $\PEW{2m-1}$ is generated by 
$\{f_1, \ldots, f_{m-1},g\}$.
\end{enumerate}
\end{prop}  

\begin{proof} (i) 
Since $0$ is the unique $\neg$-fixed point in $\Zed_{2m-1}$, every 
endomorphism of $\Zed_{2m-1}$ is also an endomorphism of 
$\W_{2m-1}$.

For (ii) we first note that $0$ is in 
the domain of any element~$h$ of $\PEW{2m-1}$ and 
$h(0) = 0$ since $h$ preserves~$\neg$.   
The only member of our standard generating set for
$\PEZ{2m-1}$ whose domain omits~$0$ is~$f_0$.

We aim 
 to show that any finite composition of partial endomorphisms
of $\Zed_{2m-1}$ which includes one or more occurrences of~$f_0$ fails to include~$0$ in its domain and so cannot be a member of  $\PEW{2m-1}$.  
  Consider  $p \circ f_0 \circ q$, where 
$p \in \PEW{2m-1}$ abs 
$q$ is a composition not including $f_0$.  Then $q(0)$ is defined and $q(0) = 0$.  But 
$p(f_0( y)) $ is undefined when $y = q(0)$.
Hence $p \circ f_0 \circ q \notin \PEW{2m-1}$.  When $p$ and/or~$q$ is absent the 
argument is even simpler.

Conversely, every element of $\{ f_1, \ldots ,f_{m-1}, g\}$ preserves~$0$ and so any composition of maps from this set
belongs to $\PEW{2m-1}$.
\end{proof}

We now turn to the even case.  Here
\cite{CPsug}*{Proposition~2.8}  tells us that 
$\PEZ{2m}$ is generated 
by the following maps:
 $h_i\colon \Zed_{2m}\setminus\{i,-i\} \to \Zed_{2m}$ 
 is defined by 
\[
h_i(a) = 
\begin{cases}
i& \mbox{if } a = i-1,\\
-i& \mbox{if } a = -(i-1),\\
a & \mbox{otherwise}
\end{cases}
\]
and 
$j\colon \Zed_{2m}\setminus\{1,-1\} \to \Zed_{2m} $ by 
\[
j(a) = \begin{cases}
a-1 & \mbox{if } a>0, \\
a+1& \mbox{otherwise.}
\end{cases}
\]
We have used different symbols here for the maps from those used in \cite{CPsug} to avoid a conflict of notation when we work with
  the variety generated by $\Zed_{2m}$ and shall need to consider $\Zed_{2m}$ and $\Zed_{2m-1}$ at the same time.

Each partial endomorphism of $\W_{2m}$ must include $\pm 1$ in its domain and  must fix these points.  Assume $m\geq 3$. For each 
element $e \in \PEZ{2(m-1)}$  we define $\overline{e}$ as follows:
\[
\overline{e} (a)  = \begin{cases}
                                  a &\text{if } a = \pm 1,\\
                                   e(a-1) + 1  &\text{if $a > 1$ and $a-1 \in \dom e$},\\
                                   e(a+1) - 1  &\text{if $ a < -1$ and $a+1 \in \dom e$}.
                              \end{cases} 
\]

\begin{prop}[partial endomorphisms of Sugihara monoids, even case] \label{prop:monpeEven}\ 

\begin{enumerate}  
\item[{\rm (i)}]
The only endomorphism of $\W_{2m}$ is the identity map.
\item[{\rm (ii)}] 
Let  $e \in \PEZ{2(m-1)}$ and define $\overline{e} $ as above.
Then $\kappa \colon e \mapsto \overline{e}$ sets up a bijection between 
$\PEZ{2(m-1)}$ and $\PEW{2m}$.  Moreover, 
$\PEW{2m}$ is generated by
$\{ \overline{h_2}, \ldots, \overline{h_{m-1}}, \overline{j}\}$. 
\end{enumerate} 
\end{prop}

\begin{proof}  Certainly $\overline{e} $ is a partial endomorphism of $\W_{2m}$ for each $e \in \PEZ{(2(m-1)}$.  
Also, every element of $\PEZ{2m}$ is injective, so the same must be true of elements of
$\PEW{2m}$.  It follows from this that $\kappa$ 
is surjective.
 The map $\lambda$ inverse to $\kappa $ acts by 
restricting $h\in \PEW{2m}$ to $\dom h  \setminus \{\pm 1\}$, relabelling domain and image in the obvious way to realise this as an element of $\PEZ{2(m-1)}$. 

The claim concerning a generating set for $\PEW{2m}$ follows from the corresponding result  for $\PEZ{2(m-1)}$ (see 
\cite{CPsug}*{Proposition~2.8}).
\end{proof}

%%%%%%%%%%%%%%%%%%%%%%%%%%%%%%%%%%%%%%%%%%%%%
%%%%%%%%%%%%%%%%%%%%%%%%%%%%%%%%%%%%%%%%%%%%%

\section{Multisorted dualities for 
Sugihara algebras}
\label{sec:alg-multi}

As promised in Section~\ref{sec:intro},
we shall present  multisorted dualities for the classes
 $\HSP (\Zed_{2m-1})=\ISP(\Zed_{2m-1})$  (the odd case)
and 
$\HSP(\Zed_{2m}) = \ISP(\Zed_{2m}, \Zed_{2m-1})$ (even case).
In the odd case our sole purpose here is to recast  our 
Piggyback Strong Duality Theorem \cite{CPsug}*{Theorem~4.3}
in two-sorted form, with a view to future applications. 
   The process we use to convert our earlier two-carrier duality for 
$\SA_{2m-1}$ could be used more generally to split a sort with more than one carrier map into separate sorts, with suitable adaptations being necessary to the alter ego.

%%%%%%%

 Our  first---and principal---task is to establish appropriate notation.  
 Fix $m \geq 2$.  
 We create two disjoint copies of   $\Zed_{2m-1}$ and call these $\P^-$ and $\P^+$.  
Let $\id_{-+}$ and $\id_{+-}$ denote,  respectively,  
 the 
 natural isomorphisms  from $\P^-$
to $\P^+$ and from $\P^+$ to $\P^-$. 
When working with $\P^-$ and $\P^+$ individually we shall
often think of each of them as equal to $\Zed_{2m-1}$.
When working with both sorts at the same time, we shall 
use superscripts to indicate  interpretations on $\P^-$ and $\P^+$.
For example,  $g^-$ and $g^+$
denote the interpretations 
 on $\P^+$ and $\P^-$ of the endomorphism $g$ of 
$\Zed_{2m-1}$.   
%%%%%%%%%%%%%%%% 

We  define  carrier maps from $\P^-$ and $\P^+$ into $\two$
as follows:
$\Omega_{\P^-} = \{ \delmv \}$ and 
$\Omega_{\P^+} = \{ \delpv\}$, where 
$$
\delpv(a) = 1 \Longleftrightarrow  a \geq 1 \text{\quad and \quad}
\delmv(a) = 1 \Longleftrightarrow  a \geq 0
$$
(cf.~\cite{CPsug}*{Section~4}).   
(Because we are working here with $\Zed_{2m-1}$ and shall later bring 
$\Zed_{2m+1}$ into the picture as well, we reserve the symbols
$\alpha^\pm$ as used with  $\Zed_{2n+1}$ in \cite{CPsug}*{Section~4} to use for carrier maps on $\Zed_{2m+1}$.)

To illustrate the use of multisorted structures we shall present in full
the check of the separation property needed in our application of
Theorem~\ref{thm:multpigstrong};
compare the following lemma with
the single-sorted version, \cite{CPsug}*{Lemma~4.1}.  

\begin{lem}[separation lemma for two-sorted duality for 
$\SA_{2m-1}$] \label{lem:2sortssep}
Let $\M \in \{ \P^+, \P^-\}$.  
Let $a, b \in \M$ with $a \ne b$.  Then there exists 
$
\M' \in \{ 
\P^+, \P^-\}  $ and  a homomorphism $\zeta \colon \M \to \M'$
such that $\w_{\M'} (\zeta(a)) \ne \w_{\M'} (\zeta(b))$.
Here $\zeta$ is either the identity map on one of the sorts or is 
a composite of maps drawn from $\{ g^-, \id_{-+}, \id_{+-}\}$.
\end{lem}

\begin{proof}  %By Lemma~\ref{lem:entail-multiodd}, 
We may without loss of generality 
assume that $\M = \P^-$.  

Suppose first that $a < 0\leq b$.  Let $\M' =\P^-$ and 
$\zeta=\id_{\P^-}$.  
Next suppose that $a < b \leq 0$.  Let $\M' = \P^-$ and 
$\zeta = (g^-)^{-b}$.  Then $\delmv(\zeta(a))  =0 \ne 1
= \delmv(\zeta(b))$.

 Now take $a 
\leq 0 < b$.  Let $\M' = \P^+$ and $\zeta= \id_{-+}$.  
Then $\delpv (\id_{-+} (a)) = 0 \neq 1 = \delpv(\id_{-+} (b))$.
Finally  suppose  $0\leq a < b$.  Again let $\M' = \P^+$ 
and let $\zeta = (g^+)^a\circ \id_{-+}$.  Then 
$\delpv( (\zeta (a))=0\ne 1=
\delmv ( \zeta (b))$.  
\end{proof}

%%%%%%%%%%%%%%%%%%%%%%%%%%%%%%%%%%%%%%%%%%%%

Our next task is to identify the piggyback subalgebras.
In \cite{CPsug}*{Proposition~4.2}, for the 
single-sorted case, we showed that
every piggyback relation (maximal or not) 
is the graph of a partial endomorphism of $\Zed_{2m-1}$ or the 
converse of such a graph.  Switching to  the two-sorted version, 
nothing changes except that we replace our piggyback subalgebras of 
$\Zed_{2m-1}^2$   by  the corresponding subalgebras
of $\M \times  \M'$, where $\M,\M' \in \{\P^-,\P^+\}$, with the appropriate carrier map acting on each coordinate. For that we need to set the following notation,
%The set of partial homomorphisms from $\M$ into~$\M'$ is denoted  by  $\PHOM(\M,\M')$, for $\M = \P^-$ and $\M'=
%\P^+$, and vice versa. 
given $\M, \M' \in \{ \P^-, \P^+\}$ we let $\PHOM(\M,\M')$
denote the set of $
\SA$-morphisms~$h $ with $\dom h \subseteq \M$ and $\img h \subseteq \M'$.

\begin{prop}[multisorted piggyback relations for $\SA_{2m-1}$]
\label{prop:mult4.2}\

\begin{enumerate}
\item[{\rm (i)}] 
A subalgebra of 
$(\delmv,\delmv)^{-1}(\leq)$ is the graph of some $h \in \End_{\text{\rm p}}(\P^-)$.

\item[{\rm (ii)}] 
A subalgebra of 
$(\delpv,\delpv)^{-1}(\leq)$ is the converse of the graph of some $k \in \End_{\text{\rm p}}(\P^+)$.

\item[{\rm (iii)}] 
A subalgebra of 
$(\delmv,\delpv)^{-1}(\leq)$ is the graph of some $h \in \PHOM(\P^-,\P^+)$
for which  $0 \notin \dom h$ and $0 \notin \img h$.

\item[{\rm (iv)}] 
A subalgebra of 
$(\delpv,\delmv)^{-1}(\leq)$ is  the graph of some 
$h \in \PHOM(\P^+,\P^-)$ or is the converse of the graph of some 
$k \in \PHOM(\P^-,\P^+)$.

\end{enumerate}
\end{prop}

Proposition~\ref{Prop:entail-multiodd} will
 allow us to avoid `doubling up' of  mirror-image structure, %at the cost of including 
since we have included 
the linking isomorphisms between the sorts.  This idea will enable us 
in due course to 
streamline our multisorted alter ego for $\Zed_{2m-1}$.
We  opt to let the sort $\P^-$ bear the brunt of  structuring  the alter ego.   But there is complete symmetry between the sorts and  we could equally well
have prioritised $\P^+$.

\begin{prop}[entailment of partial endomorphisms and  homomorphisms, odd case]
\label{Prop:entail-multiodd}
Any map in 
\[
\text{\rm (i)} \ \End_{\text{\rm p}}(\P^-),
\quad 
\text{\rm (ii)}  \ \End_{\text{\rm p}} (\P^+), 
\quad 
\text{\rm (iii)} \ 
\PHOM (\P^-, \P^+), 
\quad
  \text{\rm (iv)} \ \PHOM(\P^+,\P^-)
\]
is obtained by composition by a generating set for $\PEZ{2m-1}$,
whose members are interpreted on $\P^-$, together with 
$\{ \id_{-+}, \id_{+-}\}$.
\end{prop}

\begin{proof}  
%We exploit the fact that composition of partial
%maps is an  admissible entailment construct in the context of
%multisorted duality theory.  
The claim in the proposition for 
a map of type  (i) is immediate.  A map of type
(ii), (iii) or (iv) is, respectively,  
expressible in the form   $\id_{+-} \circ e^-
 \circ  \id_{+-}$, 
 $\id_{+-} \circ e^-$ or 
$e^- \circ  \id_{+-}$, where $e^- \in \End_p(\P^-)$.
\end{proof}

We now present our promised two-sorted strong duality theorem for 
$\SA_{2m-1}$.  As anticipated, it has a clear relationship to its single-sorted, multi-carrier, counterpart.  In particular we incorporate into the alter ego (a copy of) the same generating set  for $\PEZ{2m-1}$ as 
we used in \cite{CPsug}*{Theorem~4.3};   see Section~\ref{sec:prelim}, 
recalling \cite{CPsug}*{Proposition~2.9}.

\begin{thm}[Two-sorted Piggyback Strong Duality Theorem
for $\SA_{2m-1}$]     \label{thm:OddMulti}
Fix~$m$ with $m\geq 2$.  Take disjoint sorts $\P^-$ and $\P^+$ each   isomorphic  to $\Zed_{2m-1}$.
  Let $\Omega_{\P^-} = \{ \delmv\}$ and $\Omega_{\P^+} = \{\delpv\}$.  Consider the two-sorted 
structure based on  $\CM := \{ \P^-, \P^+\} $ having universe
$M_0 : = P^- \du P^+$.  
Let 
\[
G = \{  g, \id_{-+},  \id_{+-}\},\qquad 
H= 
\{   f_0, f_1,f_2, \ldots, 
 f_{m-2} \}, \qquad 
K = \{\boldsymbol 0_{\P^-}\}.
\]
\upshape{(}Here $g, f_0,\ldots, f_{m-2}$ are to be interpreted on $\P^-$.\upshape{)} 
 Then $\CMT := (M_0; G, H, K, \Tp) $ yields a strong and hence full duality on $\SA_{2m-1}$.   
\end{thm}

 \begin{proof}  We proceed in the same way as in the proof of \cite{CPsug}*{Theorem~4.3}, but invoking Theorem~\ref{thm:multpigstrong}
instead of the single-sorted, multi-carrier Piggyback Strong Duality Theorem given in \cite{CPsug}*{Theorem~3.3}.  
  We know that every non-trivial subalgebra of each of
$\P^+$ and $\P^-$ is subdirectly irreducible.  
Lemma~\ref{lem:2sortssep} established the separation condition.
Finally we want to invoke the Shift Lemma as given in~\ref{lem:multi-shift} 
to confirm that 
our chosen, slimmed down,  alter ego strongly dualises 
$\SA_{2m-1}$.  Propositions~\ref{prop:mult4.2} and~\ref{Prop:entail-multiodd} provide all the facts we need.
               \end{proof}

%%%%%%%%%%%%%%%%%%%%%%%%%%%%%%%%%%%%%%%%%%%%%%%%
%%%%%%%%%%%%%%%%%%%%%%%%%%%%%%%%%%%%%%%%%%%%%%%%

% 

Our next goal is
%We next  provide
 a three-sorted duality for the variety 
$\CVA{2m} =\ISP(\Zed_{2m}, \Zed_{2m-1})$,  for 
$m \geq 2$.  
Formally our sorts will have disjoint(ified) universes.  
We 
include two copies,  $\P^-$ and $\P^+$, of $\Zed_{2m-1}$ and a single copy $\Q$ of $\Zed_{2m}$.   
We shall sometimes treat the sorts as though identified with 
the appropriate $\Zed_k$, being more explicit  when this 
is warranted.

In the odd case 
we introduced the maps $\id_{-+}$ and 
$\id_{+-}$ so as to avoid including in our two-sorted alter ego
sets of partial homomorphisms  $\PHOM (\M, \M')$
for all four choices of $\M, \M'$ from $
\{\P^-,\P^+\}$.  
Likewise, with the third sort $\Q$ now in play we  want to use linking maps between $\Q$ and $\P^\pm$ to avoid 
our alter ego needing to include partial homomorphisms
going between  $\Q$ and the other sorts.  
We already in Section~\ref{sec:prelim} made use of the surjective 
homomorphism $u \colon \Zed_{2m} \to \Zed_{2m-1}$. 
This gives rise to  surjective homomorphisms  $u^\pm \colon 
\Q \to \P^\pm$.  
When we restrict $u^-$ to $\Q \setminus
\{ \pm 1\}$ we obtain a bijective homomorphism mapping onto 
$\P^-  \setminus \{0\}$.   We denote its inverse by~$v^-$.
Corresponding claims hold for $u^+$.  
See Figure~\ref{fig:sortsmaps}.  
 
\begin{figure}[ht]
\begin{center}
\begin{tikzpicture}[scale=0.8, baseline= (a).base] %tikzpicture allows rescaling of whole diagram
\node[scale=1] (a) at (0,0){
\begin{tikzcd}[row sep = .3cm, column sep = .8cm]
\Zed_{2m-1} \cong \P^- \arrow[rrrrrr, yshift=2pt , "\id_{-+}"]
 & &  & &  &    & \P^+  \cong\Zed_{2m-1}
\arrow[llllll,  yshift=-2pt,  "\id_{+-}"]\\  
\P^- \setminus \{ 0\}  \arrow[rrrdddd, xshift=-2pt, swap, "v^-"]
 \arrow[u, hook, yshift=-1pt]  &  & &  & & &
 \P^+ \setminus \{ 0\} \arrow[llldddd, xshift=2pt, "v^+"] 
\arrow[u, hook', yshift=-1pt]\\
&&&&&&\\   %row3
& & &  & & &\\ %row4
%&&&\Q &&&\\
%&&&\Q &&&
&  & & \Q  \cong \Zed_{2m} \arrow[llluuuu, swap, "u^-"]  \arrow[rrruuuu, "u^+"] &&& \\
%&  & & \Q  \arrow[llluuuu]  \arrow[rrruuuu]
%\arrow[xshift=4pt, rrrdd] \arrow[xshift=-4pt, llldd] 
%& & &\\ %row5
& & & \Q \setminus \{ \pm 1\}   \arrow[u, hook]  & & & %row 6
\end{tikzcd}
};
\end{tikzpicture}
\end{center}
\caption{Linking maps between sorts, even case \label{fig:sortsmaps}}
\end{figure}
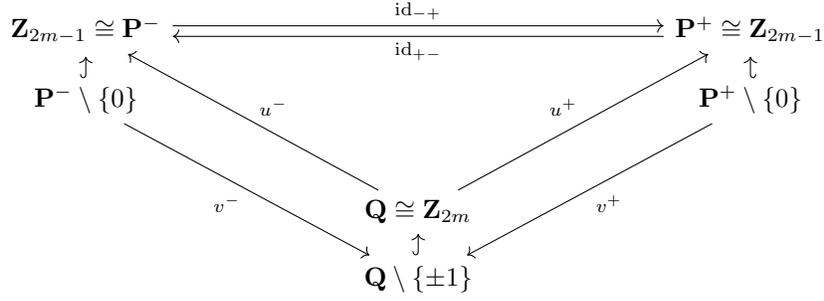

%%%%%%%%%%%%%%%%%%%%%%%%%%%%%%%%%%%%%%%%%%%%%

%
%%%%%%%%%%%%%%%%%%%%%%%%%%%%%%%%%%%%%%%%%%

%%

We now set up the carrier maps we shall need in order to satisfy 
the separation condition in the Piggyback Theorem.  
As for the odd case we take 
\begin{alignat*}{3}
&\text{from }\fnt{U}(\P^-)\to \two:  \qquad \qquad     
&& 
\delpv (a) = 1 \Longleftrightarrow  a\geq 1;\qquad\qquad \qquad \qquad &&\\
&\text{from }\fnt{U}(\P^+)\to \two:  \qquad 
&& 
\delmv (a) = 1 \Longleftrightarrow  a \geq   0; \\
\shortintertext{and, from the even case as analysed in 
\cite{CPsug},} 
&\text{from }  \fnt{U}(\Q)\to \two:
&&\beta 
(a) = 1 \Longleftrightarrow a > 0.
\end{alignat*}
We take 
$\Omega_{\P^-} = \{ \delmv\} $, $\Omega_{\P^-}= \{\delmv\}$,  and $\Omega_{\Q} = \{\beta\}$.  
For $\M \in\{\P^-,\P^+, \Q\}$,  we   write $\w_{\M}$ for 
the unique element of $\Omega_{\M}$.

\begin{lem}[separation lemma for $\CVA{2m}$]  \label{lem:3sortssep}
Define  $ g^-, \id_{-+}, \id_{+-}, u^-$ as above. 

Let $\M \in \{ \P^-, \P^+, \Q\}$.
and 
$a \ne b$ in $\M$.  Then there exists $\M' \in 
\{ \P^-, \P^+. \Q\} $ and $\zeta \in \HOM (\M, \M')$ such that
 $\w_{\M'}  (\zeta (a))  \ne \w_{\M'}  (\zeta(b))$, where $\zeta$ is an   identity map on one of the sorts or is a composite of maps drawn from $\{ g^-, \id_{-+}, \id_{+-}, u^-\}$. 
\end{lem}

\begin{proof}  
Lemma~\ref{lem:2sortssep} covers the cases in which $\M \in 
\{\P^-,\P^+\} $.
So  
let 
$c\ne d $ in $\Q$.   If $c$ and $d$ are of opposite sign then $\beta(c) \ne \beta (d)$ and we take $\M'=\Q$  and $\zeta =\id_{\Q}$. 
 Now take $c < d \leq - 1$ in~$\Q$.  Then
$u^- (c) <  u^-(d) \leq 0$ in  $\P^-$.   
Let $a := u^-(c)$ and $b := u^-(d)$ and proceed as in the proof of
Lemma~\ref{lem:2sortssep}. 
Similarly,  if $c > d \geq 1$ in $\Q$ then 
$u^+(c) > u^+(d) \geq 0$ in $\P^+$.  Again we can now argue   
as in the proof of Lemma~\ref{lem:2sortssep}.  
\end{proof}

%%%%%%%%%%%%%%%%%%%%%%%%%%%%%%%%%%%%%%%

We must now describe the multisorted piggyback relations for $\CVA{2m}$.
Given $\M, \M' \in 
\{ \P^-, \P^+, \Q\}$, we shall write
$
R_{\M,\M'}
$ for the set of universes $r$ of subalgebras $\mathbf r$
of 
$\M \times \M' $ such that  
$ r \subseteq
(\w_{\M}, \w_{\M'})^{-1}(\leq)\}$ and $\mathbf r$ is maximal 
with  respect  to this containment.  

We already know that every piggyback relation in the odd case 
(maximal or not) is 
the graph of a partial homomorphism or the converse of such a 
graph. In addition, every relation in $R_{\Q,\Q}$ 
is the graph of a member of $\End_p(\Q)$, by \cite{CPsug}*{Proposition~6.3}.   
We shall subsequently want to present a duality for the Sugihara 
monoid variety  $\CVM{2m}$ by making small
adaptations to that for $\CVA{2m}$.  
This means that it is expedient when we encounter graphs of $\SA$-morphisms
in the algebra setting we should be sufficiently explicit to be able  detect easily 
whether  $\SM$-morphisms are also available.  Furthermore 
our structural analysis of Sugihara algebras and monoids in succeeding 
papers requires detailed information on piggyback relations and it is convenient to record this now.

\begin{table}[ht]
 \begin{center}
\begin{tabular} {l|ccccc}
                      &  $ \phantom{n}\P^-$ \phantom{n}& & $\phantom{n}\Q 
\phantom{n}$   & & $\phantom{n}\P^+ $
\phantom{n}
 \bigstrut[b]  \\[1ex]
\hline
&&&& \\[-1.5ex] 
$\P^-  $  &  $ \grh $  &
\ \   & 
 $\grh$ 
 &
\ \     
&  $\grh $\bigstrut[t]  \\
&&&&& with $0 \notin \dom h, \, 0 \notin  \img h$ 
                  \\[2ex]
$\Q$  &  $\grh$   &  &
             $\grh  $ 
 & & $\grh$,
with $0 \notin \img h $\\[2ex] 
%hp2002  best layout for the P+ row?  and for P- row.   ?? use multirow??
$\P^+$  
& \quad $\grk^\smallsmile$  or  & & \!\!\! $\grk^\smallsmile$, with 
$0 \notin \dom k$ & &  $\grk^\smallsmile$
\\
& \quad $\grh $  
& &
  
 \end{tabular}
\end{center}
\medskip

\caption{Piggyback relations for $\CVA{2m}$ \label{table:CVApigrels}}
\end{table}

%%%%%%%%%%%%%%%%%%%%%%%%%%%%%%%%%%%%%%%%%%%

\begin{prop}[multisorted piggyback relations for $\CVA{2m}$] \label{prop:varalg-allrels}
 Let $m \geq 2$. 
 For each choice of $\M,\M' \in 
\{ \P^-, \P^+, \Q\}$,  the entry 
in the $\M,\M'$-cell in
Table~{\rm{\ref{table:CVApigrels}}}  shows the form that a member 
$S $ of  $R_{\M,\M'}$ must take.  If $\grh $ appears in that cell  it is to be assumed that $h$ 
 is a partial homomorphism with
$\dom h \subseteq  \M$ and $\img h \subseteq \M'$, with 
additional properties  as stipulated. 
In cases in which $ \grk^\smallsmile$ appears then $k$ is to be taken to be a homomorphism with $\dom k \subseteq \M'$
and $      
\img k \subseteq \M$.   Necessary restrictions arising from $0 \notin \Q$ are left implicit.
  \end{prop}

\begin{proof}  The entries in the four corner cells 
in the table deal with the forms taken by the 
piggyback relations in the sets $R_{\M, \M'}$ 
for which  neither $\M$ nor $\M'$ is $\Q$; see Proposition~\ref{prop:mult4.2}.
The results originate in   \cite{CPsug}*{Proposition~4.2}, applied with $n  = m-1$.   The dichotomy in the entry in the bottom left corner  is explained in  item (iv) of that proposition. 
 To describe  the members of $R_{\Q,\Q}$ we call on  \cite{CPsug}*{Proposition~6.3}:
any (maximal) piggyback subalgebra is the graph of a (necessarily
invertible) partial endomorphism. 

We now deal with the four remaining  cases,
in which we need to consider sorts of different parities.

Throughout we work with $S \in R_{\M,\M'}$ 
initially viewed 
as a subalgebra of 
$\Zed_{2m+1}^2$,  with $\Q$ identified with $\Zed_{2m}$ and whichever of 
$\P^-$ or $\P^+$ is in play in a particular item identified with
$\Zed_{2m-1}$; in this setting we regard $\Zed_{2m}$ and 
$\Zed_{2m-1}$ as subalgebras of $\Zed_{2m+1}$.    
 We may  then regard 
$\w_{\P^-} =\delmv$ and $\w_{\P^+}= \delpv$, 
as restrictions to $\Zed_{2m-1}$ of the maps  $\delm$  and $\delp$ used as
carrier maps  for the duality for $\SA_{2m+1} $ as %given in \cite{CPsug}*{Section~4}:
$
\delp 
(a) = 1 \Leftrightarrow  a \geq   1$  and $
\delm 
 (a) = 1 \Leftrightarrow a \geq 0$.
Our strategy for classifying   piggyback relations linking  $\Q$ 
with the sorts $\P^-$ and $\P^+$ will  be first to 
realise them, up to isomorphism,   as subalgebras of~$\Zed_{2m+1}^2$  associated with  appropriate piggyback relations for the odd case.   We then wish to assert that the 
piggyback relations from which we started are graphs of partial
homomorphisms or converses thereof, where 
We  take account of any  constraints inherent in the various cases as regards domain and image of the partial maps which can arise. 
To justify our claims  two approaches  are available.
 We can  appeal directly 
 to \cite{CPsug}*{Proposition~4.2}  and then switch into the language of sorts
or alternatively   we can translate to  sorts first and then call on Proposition~\ref{prop:mult4.2}.  The approaches are equivalent but, either way, some obvious re-alignment of notation is required.

%%%%%%%%%%%%%%

\begin{widenewlist}
\item[$R_{\P^-,\Q}$:]  
  %%%%DOWN
\enspace  
Consider  a subalgebra  $S$ of $
%\P^- \times \Q$ 
\Zed_{2m-1} \times \Zed_{2m}$  
contained in
$ (\delmv, \beta)^{-1}(\leq)$.   
 Then 
$(a,b) \in S$ and $\delmv(a) = 1 $ together imply $\beta(b) = 1$.
This implies that $\delm(a) = 1$ forces $\delm(b) = 1$ since 
$b=0$ does not arise.  
%By \cite{CPsug}*{Proposition~4.2(i)}, 
Then 
 $S$ is the graph of a partial endomorphism~$e$ of $\Zed_{2m+1}$,  where $\dom e$  omits $\pm m$ and $\img  e $ omits $0$. Translating into the language of sorts, we 
obtain  the characterisation of $R_{\P^-,\Q}$ shown in the table.

\item[$R_{\Q,\P^-}$:]
\enspace 
Let $S$ be a subalgebra  of   $\Zed_{2m} \times\Zed_{2m-1}$
%$\Q \times \P^-$
 for which $S \subseteq (\beta, \delmv)^{-1}(\leq)$.
This time we may view $S$ as a subalgebra of 
$\Zed_{2m+1}^2$ within $(\delp, \delm)^{-1}(\leq)$;   necessarily, 
$S \cap \bigl (\{ 0\} \times \Zed_{2m} \bigr )= \emptyset$. 
We deduce %from \cite{CPsug}*{Proposition~4.2(iv)}
that 
 $S \subseteq (\delm, \delm)^{-1}(\leq)$.   As such, $S$  is the graph
of a member of $\PEZ{2m+1}$.  
Restating in terms of sorts, we obtain the required result.

\item[$R_{\Q,\P^+}$:]   %%%%DOWN
\enspace 
Let $S$ be a subalgebra of %$\Q \times \P^+$
$\Zed_{2m} \times \Zed_{2m-1}$ 
for which $S \subseteq (\beta, \delpv)^{-1}(\leq)$.
This time $(a,b) \in S$  implies $a \ne 0$ so that $S$,
regarded as a subalgebra of $\Zed_{2m+1}^2$, is contained in 
$(\delm, \delp)^{-1}(\leq)$.  
Any allowable subalgebra is the  graph  of a partial endomorphism $e$ of $\Zed_{2m+1}$ which excludes~$0$ from both domain and image.  In terms of sorts, this gives the graph of a partial homomorphism whose image excludes~$0$.

\item[$R_{\P^+,\Q}$:]  %%%%UP  DONE
\enspace  
Let  $S$ be  a subalgebra of %$\P^+ \times \Q$ 
$\Zed_{2m-1} \times \Zed_{2m}$ 
contained in $(\delpv, \beta)^{-1}(\leq)$.  
View $S$ as a subalgebra of 
$\Zed_{2m+1}^2$.  
Here $(a,b) \in S$ implies 
$a \ne \pm m$, 
$b \ne 0$
and $\delpv(a) =1$  forces $\beta (b) = 1$.  But for such $a, b$ 
this last condition is the same as $\delp(a) = 1 $ forces 
$\delp(b) = 1$. 
Hence $S^\smallsmile $ is the graph of a partial
endomorphism~$k$ of $\Zed_{2m+1}$,  by \cite{CPsug}*{Proposition~4.2(ii)}.  We require $0 \notin \dom h$
(this is of course necessary) and $\img h \in \Zed_{2m-1}$.  So
$k$ can be viewed as a partial homomorphism from $\Q$ into $\P^+$.  \qedhere
\end{widenewlist}
\end{proof}

A few additional remarks on the entries in Table~\ref{table:CVApigrels} are in order.  First of all there is more symmetry than may at first sight appear.  The 
elements of $R_{\Q,\Q}$ are graphs of partial endomorphisms
$h$ of $\Q$.  Any such~$h$ is invertible and $k:= h^{-1}$ is such that $\grk ^\smallsmile = \grh$.    
When we translate from $\Zed_{2m+1}^2$ into partial maps
between sorts, presence or absence of~$0$ from domain
or image is automatic.  

We shall derive an analogue   
of Proposition~\ref{Prop:entail-multiodd}, now also dealing with entailment of partial 
homomorphisms between non-isomorphic sorts.  
Refer back to the $\CM$-Shift Strong Duality Lemma~\ref{lem:multi-shift} to see the
relevance of the conditions.  Those familiar with the  requisite 
technicalities can rephrase  
the claims in Proposition~\ref{prop:varalg-entailphoms}  in terms of hom-entailment.

%%%%%%%%%%%%%%%%%%%%%%%

% 

\begin{prop}[partial homomorphisms between sorts] \label{prop:varalg-entailphoms}
\

\begin{enumerate}
\item[{\rm (i)}] Any  $h \in \PHOM (\P^-, \Q)$ is 
expressible as a composite of maps drawn from 
$\End_p(\Q) \cup \{ v^-\}$. 

Any  $h' \in \PHOM (\P^+, \Q)$
is  
expressible as a composition of maps drawn from
$\End_{\text{\rm p}}(\Q) \cup \{ v^-\} \cup 
\{ \id_{+-}\}$.

\item[{\rm (ii)}] 
Any  $h \in \PHOM (\Q,
\P^-)$ 
expressible as the restriction of a  composition of maps drawn from $\End_{\text{\rm p}}(\P^-) \cup
\{u^-\}$.

Any  $h' \in \PHOM (\Q, \P^+)$ is expressible as the restriction 
of a composition of maps drawn from 
$\End_{\text{\rm p}}(\P^-) \cup
\{u^-\}
\cup \{ \id_{-+}\}$.
\end{enumerate} 
\end{prop}

\begin{proof} We showed in   \cite{CPsug}*{Proposition~2.6} 
that 
for any $p \leq n$ and for $ 0 <  a_1 < \cdots < a_p$ and 
$0 < b_1 < \cdots < b_p$ there exists $e \in \PEZ{2n}$
 such that $e (a_i)= b_i$ for $1 \leq i \leq p$.  Moreover
$e$ extends to a partial endomorphism of $\Zed_{2n+1}$ which sends~$0$ to~$0$.    
Necessarily $a \in \dom e$ if and only if $-a \in \dom e$.
We may assume that $|\dom e| = 2p$ in the odd case
and $2p+1$ in the even case. 
 
Consider (i).  
Take  $h \in \PHOM( \P^-, \Q)$.   Necessarily 
$0 \notin \dom h$ (since $\Q$  has no $\neg$%
fixed point).
Then $h$ is injective  and so $|\dom h| = |\img h|$.  
(See 
\cite{CPsug}*{Proposition~2.2}.)  But   $v^-(\dom h)$ is also a subalgebra 
of $\Q$ of cardinality 
$|\dom h|$.    Moreover, 
both~$h$ and $v^-$ are strictly monotonic.  Hence  
there exists $e \in 
\End_p (\Q)$ with 
$h (a)= e (v^-(a)) $ for $a \in \dom h$ and  
$\dom (e \circ v^-) = \dom h$.  
Therefore $ h= e \circ v^-$.  This proves the first claim in~(i).

Consider $h'\in \PHOM(\P^+,\Q)$.
  From above,
$h' \circ \id_{-+}  =  e \circ v^-$ for some $e \in \End_p (\Q)$.
 Now observe that  $h' =  e \circ v^- \circ \id_{+-}$. 
This completes the proof of~(i).

Now consider (ii).
 Take  $h\in \PHOM(\Q,\P^-)$. First assume $0 \notin \img h$.  
  Then $h$ maps $\dom h$ 
into $\P^\pm \setminus \{ 0\}$ and \cite{CPsug}*{Proposition~2.2} implies that $h$ is injective
and  
$u^-{\restriction}_{\dom h}$ is also injective.    
Then there exists $e\in
\End_p(\P^-)$ such that 
$\dom e = u^-(\dom h) $ and $h(a)  = e(u^-(a))$ for all 
$a \in \dom h$.  Thus  $h=(e\circ u^-){\restriction}_{\dom h}$.

Now assume that  $0 \in \img h$. 
  Then there exists a  minimal 
$s > 0$ such that $h^{-1}(0) \subseteq  [-s,s] = (u_s^-)^{-1}(0)$.  
Let  $h_s :=   (g^-)^{s-1}  \circ u^-$,  Then $h_s$ is a homomorphism from~$\Q $ into $\P^-$.  
 Both $h$ and 
$h_s$
are strictly monotonic on
$\{ s+1, \ldots, m\}$.  
There exists  a partial
endomorphism $e$ of $\P^-$  with $\dom e = \img h$ and  
$h(a) = 
e(h_s(a))$ for $a \in \dom h$ and $a > s$.  Since $\dom h \cap
(h_s)^{-1}(0) \ne \emptyset $ and  $0 \in \dom e$
the maps $e \circ h_s$ and $h$ 
agree on the domain of $h$. 
This completes the proof of the first assertion in~(ii).

For the second assertion we consider  $\id_{+-} \circ h'$.   
\end{proof}

We have  assembled all the ingredients  for our duality theorem 
for $\CVA{2m}$. 
Table~\ref{table:varalg-alterego} indicates the maps we shall
include in our multisorted alter ego.
In the table all undecorated maps are to be viewed as being 
interpreted as maps between the indicated sorts.  That is, we have 
omitted the superscripts previously used to indicate 
the sorts when interpretations of $\SA$-morphisms between 
$\Zed_{2m}$. $\Zed_{2m-1}$, and their subalgebras, are involved.
(The position of each entry
dictates  the intended interpretations of the maps.)

\begin{table} [ht]
\begin{center}
%\begin{table}
%\begin{center}
\begin{tabular} {l|ccccc}
                      &  $ \phantom{mm}\P^-$ \phantom{mm}& & $\phantom{mm}\Q 
\phantom{mmm}$   & & $\phantom{mm}\P^+ $
\phantom{mm}
 \bigstrut[bt]  \\[0.5ex]
\hline\\[-1.5ex]
$\P^-  $  &  $  f_1, \ldots,f_{m-2};  \ g $  &  
\ \   & $v$
  %, with   $0 \notin \dom h  $  
 &
\ \     
&  $\id_{-+} $\\
&\small{[as per odd case]}&& \small{[linking partial } &&  \small{[linking isomorphism]}
\\[-1ex]
&&& \small{isomorphism]} &&
\bigstrut[t]   
                  \\[2ex]

$\Q$  &  $u$   &  & $h_2, \ldots, h_{m}, j$ &
            &  {\bf ---}   \\
 & \small[{surjective homomorphism]}   & &  \small{[as per even case]}\\[2ex] 
%  &  
%%\begin{cases}  
%$ \grh$ ,  
%disjoint from &&&\\
%&  $ (\P^- \setminus \{ 0\} )\times \{0\}$,  &&&\\
$\P^+$ &  $\id_{+-}$
&&{\bf ---}  
  &  &{\bf ---}   \\
& \small{[linking isomorphism]}  
&&&& 
\end{tabular}

\medskip 

\caption{Maps to be included in the alter ego for $\CVA{2m}$
\label{table:varalg-alterego}}
\end{center}
\end{table}

Our duality theorem is proved in essentially the same way as that for the odd case,
 and we need only give a sketch of the proof.

%%%%%%%%%%%%%%%%%%%%%%%%%%%%%%%%%%%%%%%%%%%
%%%%%%%%%%%%%%%%%%%%%%%%%%%%%%%%%%%%%%%%%%%%%%%%%

\begin{thm}[Multisorted Strong Duality Theorem for 
$\CVA {2m}$]\label{thm:EvenMulti}  %5.8
Let  $\P^-$,  $\P^+$ and $\Q$ be  disjointified copies of 
$\Zed_{2m-1}$, $\Zed_{2m-1}$ and $\Zed_{2m}$, respectively, with carrier maps $\delmv$, $\delpv$ and $\beta$, respectively.
Let 
$\CM= \{ \P^-, \P^+,\Q\}$ and let $\Tp$ be the disjoint union topology 
on  $M_0 := P^- \du P^+\du Q$ derived from the discrete
topology on each sort.   Let $
\CMT= (M_0;  G,H,K,\Tp)$, where the elements of $G\cup H$
are the maps presented in Table~\rm{\ref{table:varalg-alterego}} and $K = \{\boldsymbol 0_{\P^-}\}$.  Then $\CMT$ strongly dualises 
$\CVA{2m}$.
\end{thm}

\begin{proof} With our selected sorts and carrier maps, the 
proposed alter ego ensures that 
the
 separation condition is satisfied; see
 Lemma~\ref{lem:3sortssep}. 
We now consider an enlarged alter ego in which we include 
all the structure that Theorem~\ref{thm:multpigstrong} tells us will give a strong duality,  Then we can discard from this
relational structure any operations,  partial operations, relations or constants which are entailed by the structure we retain.  
 Proposition~\ref{prop:varalg-allrels} implies that all the piggyback relations are entailed by the full sets of partial endomorphisms and  partial  homomorphisms which we have included in our enlarged alter ego to guarantee strongness.  
%lc1801: Added
Observe that $f_0\colon \P^-\setminus\{0\}\to\P^-\setminus\{0\}$ coincides with $v^-\circ h_m\circ \cdots\circ h_2 \circ u^-$.  Hence, by \cite{CPsug}*{Propositions~2.8 and~2.9}, every element of $\End_p(\P^-)$ and $\End_p(\Q)$ can be obtain as a composition of maps in $G\cup H$.  
%%%%
Finally we call on 
Proposition~\ref{Prop:entail-multiodd} and Proposition~\ref{prop:varalg-entailphoms}.
We conclude that, with  $G \cup H \cup K$
as in our  proposed alter ego, we indeed obtain a  strong duality for $\CVA{2m}$.  
\end{proof}

We have made our duality economical at the price of a loss of symmetry
as regards $\P^-$ and $\P^+$.  When working with the duality 
it is likely to be convenient to think in terms of a symmetric formulation.
However
 it is worth noting that we cannot 
thereby dispense with the linking maps $\id_{-+}$ and $\id_{+-}$ since 
these are needed for the separation lemma. The same comment applies to the duality for~$\SA_{2m-1}$.

We should expect some asymmetry in relation to the  piggyback relations,  in  Theorems~\ref{thm:OddMulti} and~\ref{thm:EvenMulti},
and in their single-sorted
 counterparts in \cite{CPsug}.  This stems from the way the carrier maps 
operate.   Hints of this feature will be visible in our discussion
of Kleene algebras in Section~\ref{sec:why-multi} and it will
come to the fore when we later use our duality to  
describe free algebras \cite{CPfree}.

We draw attention to the omission of $f_0$ from $\PEZ{2m-1}$
that was made possible by the presence of the (necessary) map
$v^-$ which links $\P^-$ and~$\Q$.  This aside,
 the duality theorem for
$\CVA{2m}$
  encodes within it all the information for the duality theorem for 
$\SA_{2m-1} = \CVA{2m-1}$. We just delete all reference to the sort $
\Q$ and maps which involve it, and insert $f_0$ to  recapture all the information
required for  Theorem~\ref{thm:OddMulti}.

%%%%%%%%%%%%%%%%%%%%%%%%%%%%%%%%%%%%%%%%%%%%%%%
%%%%%%%%%%%%%%%%%%%%%%%%%%%%%%%%%%%%%%%%%%%%%%%

\section{Multisorted dualities for
Sugihara monoids}
\label{sec:mon-multi}

We now consider the monoid case, referring  back to 
Section~\ref{sec:prelim} for the basic algebraic results. 
We shall head  straight for a multisorted
duality for the variety $\CVM{2m}$, encompassing
$\SM_{2m-1}$ along the way rather than separating it out.
In all other respects,
%%%%%%
 we shall structure  our argument in the same way as
the corresponding argument in Section~\ref{sec:alg-multi}:
%%%%
sorts and linking maps; 
carrier maps and separation; piggyback relations; 
entailment of maps on and between sorts; 
the strong duality theorem for $\CVM{2m}$.
Throughout we 
 highlight where adaptations do and do not need to be made to transition
from~$\SA$ to~$\SM$.
Let $\fnt V\colon \SM \to \SA$ denote the natural forgetful functor.
Thus $\fnt V(\W_k ) = \Zed_k$ for each~$k$.

We replace the three sorts $\P^-, \P^+$ and $\Q$ 
from Section~\ref{sec:alg-multi}
by sorts 
denoted $\S^-, \S^+$ and $\T$.  The latter will be disjointified
copies of $\W_{2m-1}, \W_{2m-1}$ and $\W_{2m}$, respectively.   
We may assume that $\fnt V(\S^\pm)=\P^\pm$ and $\fnt V(\T) = \Q$.

Sugihara monoid  homomorphisms must preserve
$\mathbf t$ (and also $\mathbf f = \neg \mathbf t$).  
The isomorphisms $\id_{-+}$ and $\id_{+-}$ are 
$\SM$-morphisms between $\S^-$ and $\S^+$ and
Proposition~\ref{Prop:entail-multiodd}  carries over, \textit{mutatis mutandis},
to the monoid case.

Since $0$ in any 
Sugihara monoid is a fixed point for~$\neg$, there are no 
monoid (partial) homomorphisms from $\W_{2m-1}$ to any even Sugihara 
monoid. 
Any $\SM$-morphism from 
$\W_{2m}$
to any $\W_{2k-1}$ must send $\pm 1$ to~$0$.  
The previously-defined map $u\colon \Zed_{2m} \to 
\Zed_{2m-1}$ sends $\pm 1$ to $0$ and so also serves as
an $\SM$-morphism from $\W_{2m}$ onto $\W_{2m-1}$.  
We define $u^\pm$ in just the same way as before, but now using the monoid sorts.

We shall need an  $\SM$ analogue
 of 
Proposition~\ref{prop:varalg-entailphoms}, whereby partial homomorphisms
between sorts are  derived by composition from link maps between
sorts and generating sets for partial endomorphism monoids.
We do still have available the endomorphism~$g$ which played an important role in
the proof of Proposition~\ref{prop:varalg-entailphoms}(ii).   
A key difference however emerges with~(i).     The distinguished constant(s) must be in the domain of any monoid 
(partial) homomorphism.  
Consequently the sets
$\PHOMt(\S^\pm,\T)$ are empty, since any 
member of  $\PHOM(\P^\pm,\Q)$ excludes~$0$ from its domain.
Hence entailment does not arise.  Note that 
analogues of the maps $v,v^{\pm}$ do not exist, but are not needed since $v^-$  was used  in Proposition~\ref{prop:varalg-entailphoms} only to handle $\PHOM(\P^\pm,\Q)$.  

When  applying Proposition~\ref{prop:varmon-entailphoms} later we shall take advantage of Propositions~\ref{lem:monpe:odd}
and~\ref{prop:monpeEven}.

\begin{prop}[%
partial endomorphisms and homomorphisms]
 \label{prop:varmon-entailphoms} \

\begin{enumerate}

\item[{\rm (i)}]   %odd case 
Any member of 
 $\End_{ß\text{\rm p}} (\S^-)$, $\Endt_{\text{\rm p}}(\S^+)$, 
$\PHOMt (\S^-, \S^+)$  or $\PHOMt(\S^+,\S^-)$
is expressible as  a
 composition of maps drawn from a set obtained 
by adding  $\id_{-+}$ and $ \id_{+-}$ to a 
generating set for $\PEW{2m-1}$
whose members are interpreted on $\S^-$. 

\item[{\rm (ii)}]  Any $h \in \Endt_{\text{\rm p}}(\T)$ is entailed by a generating set
for 
 $\Endt_{\text{\rm p}}(\W_{2m})$ 
interpreted on $\T$.

\item[{\rm (iii)}]
Any $h \in \PHOMt(\T, \S^-)$ is 
expressible as a restriction of a composition of maps drawn from a set obtained 
by adding $u^-$  to 
a generating set for $\PEW{2m}$ interpreted on  $\S^-$.
     
Any  $h' \in  \PHOMt(\T, \S^+)$ is expressible as a restriction of a composition
of maps drawn from a set obtained by adding $u^-$ and
$\id_{-+}$  to a generating set for $\PEW{2m}$ interpreted on  $\S^-$. 
\end{enumerate} 
\end{prop}

Carrier maps act on the lattice reducts of the sorts 
and no change is required when we pass to Sugihara monoids.
So we take $\w_{\S^-} = \delmv$,  $\w_{\S^+} = \delpv$ and
$\w_\T = \beta$.   Lemma~\ref{lem:3sortssep} established 
the separation condition
for the algebra case.  All the maps used there are 
$\SM$-homomorphism so separation also holds for the monoid case.

We want  to identify the forms taken by (maximal) $\SM$-subalgebras 
of sublattices 
$(\w_{\M},\w_{\M'})^{-1}(\leq) $ 
of $\M,\M'$, where $\M,\M' 
\in \{\S^-,\S^+,\T\}$.

Consider first the  choices of $\M,\M'$ for  which a characterisation is available of \emph{all} $\SA$-subalgebras
of $(\w_{\M}, \w_{\M'})^{-1}(\leq)$,  where $\M,\M' \in \{ \P^-, \P^+,
\Q\}$, as a graph of an $\SA$-partial homomorphism or converse
of such a graph.  The only situation in which we identified only the
\emph{maximal} $\SA$-subalgebras of $(\w_{\M},\w_{\M'})^{-1}(\leq)$
was that in which $\M = \M' = \Q$.  
So suppose $S$ is an $\SM$-subalgebra of 
$(\w_{\M}, \w_{\M'})^{-1}(\leq)$, where $\M,\M' \in \{\S^-.\S^+, \T\}$
and $\M,\M'$ are not both~$\T$.   Then $\fnt V (S) $ 
takes the form shown  in the $\fnt V (\M), \fnt V(\M')$ cell in 
Table~\ref{table:CVApigrels}.  
This tells us that $\fnt V(S)$, or its converse, is the graph of
 an $\SA$-morphism, possibly with additional restrictions 
on domain and/or image of the map.
We must then insist that the morphism
  be an $\SM$-morphism, and consider whether any
restrictions on domain or image impose further constraints. 
Specifically, 
let $\M \in \{ 
\S^-, \S^+, \T\}$,  so $\mathbf t^{\S^\pm}  = 0$ and 
$\mathbf t^{\T} = 1$.  
For sorts $\M,\M'$ and 
 $h \in \PHOMt (\M, \M')$, necessarily $\mathbf t^{\M} \in \dom h$ and 
$\mathbf  t^{\M'} = h(\mathbf t^{\M})$.  In addition, 
a subalgebra of a  product $\M 
\times\M'$ of sorts contains   $(\mathbf t^\M,  \mathbf t^{\M'})$.
We note in particular that $(\pm 1,0)$ belongs to any $\SM$-subalgebra of $\T \times \S^\pm $.  Since 
$
(1,0) \notin 
(\delmv, \delpv)^{-1}(\leq)$ and $(1,0) \notin (\delmv, \beta)^{-1}(\leq)$ there are no
$\SM$-subalgebras of $\S^- \times \S^+$ or of 
$\T \times \S^+$.   This signals the key differences between
the dualities for Sugihara algebras and monoids, in both odd and even cases.
%%%%%%%%%%%%%%%%%%%%%%%%%%%

The $\SM$-subalgebras
of $(\beta,\beta)^{-1}(\leq)$
require special consideration.
We know that a maximal $\SA$-subalgebra of 
$(\beta,\beta)^{-1}(\leq)$ is the graph of an $\SA$-homomorphism.  But  there may be fewer 
$\SM$-subalgebras of  $(\beta,\beta)^{-1}(\leq)$
than $\SA$-subalgebras, and we have no guarantee that
a maximal element in $R_{\Q,\Q}$ will be the graph of an
$\SM$-homomorphism.  Hence we need to re-run the proof of the 
$\SA$ result in \cite{CPsug}*{Proposition~6.3 (residual case (a))} to check it works analogously when applied in the 
$\SM$ setting. 
Table~\ref{table:CVMpigrels} parallels 
Table~\ref{table:CVApigrels}.

\begin{table}[ht]
 \begin{center}
\begin{tabular} {l|ccccc}
                      &  $ \phantom{mm}\S^-$ \phantom{mm}& & $\phantom{mm}\T 
\phantom{mmm}$   & & $\phantom{mm}\S^+ $
\phantom{mm}
 \bigstrut[b]  \\[1ex]
\hline
&&&&&\\[-.7ex]
$\S^-  $  &  $ \grh $
  &
\ \   &  {\bf ---}  
 %$\gr$h
 %, with   $0 \notin \dom h  $  
 &
\ \     
&  {\bf ---} %$\grh $
\bigstrut[t]  
%\\
%&&&&& with $0 \notin \dom h, \, 0 \notin  \img h$ 
                  \\[2ex]
$\T$  &  $\grh$   &  &
             $\grh  $ %, with $h $ invertible
 & & 
 {\bf ---}
%$\grh$ with $0 \notin \img h $
\\[2ex] 
$\S^+$ &  $\grh$ or $\grk^\smallsmile $
&&   {\bf ---} 
   %$\grk^\smallsmile$, with $0 \notin \img k $
   &  & $ \grk^\smallsmile   $

\end{tabular}
\end{center}
\medskip

\caption{Piggyback relations for $\CVM{2m}$ \label{table:CVMpigrels}}
\end{table}

%%%%%%%%%%%%%%%%%%%%%%%%%%%%%%%%%%%%%%%%%%%%

We identified generating sets for the partial endomorphisms of $\W_{2m-1}$ and of $\W_{2m} $  in 
Lemmas~\ref{lem:monpe:odd} and \ref{prop:monpeEven} and we incorporate these into our alter ego for~$\CVM{2m}$.

\begin{table} [ht]
\begin{center}
%\begin{table}
%\begin{center}
\begin{tabular} {l|ccccc}
                      &  $ \phantom{mm}\S^-$ \phantom{mm}& & $\phantom{mm}\T 
\phantom{mmm}$   & & $\phantom{mm}\S^+ $
\phantom{mm}
 \bigstrut[b]  \\[1ex]
\hline
&&&&&\\[-.7ex]
$\S^-  $  &  $  f_1, \ldots,f_{m-2};\ g $  &
\ \   & 
 {\bf ---}    %, with   $0 \notin \dom h  $  
 &
\ \     
&  $\id_{-+} $\\
&\small{[as per odd case]}&&&&  \small{[linking isomorphism]}\bigstrut[t]   
                  \\[2ex]
$\T$  &  $u$   &  & $\ov{h_2}, \ldots , \ov{h_{m-1}}, \ov{j}$ &
            &    {\bf ---}  \\
 & \small[{surjective homomorphism]}   & &  \small{[as per even case]}\\[2ex] 
$\S+$ &  $\id_{+-}$
&& {\bf ---}  
  &  &  {\bf ---}  \\
& \small{[linking isomorphism]}  
&&&& 
\end{tabular}
\medskip

\caption{Maps to be included in the alter ego for $\CVM{2m}$
\label{table:varmon-alterego}}
\end{center}
\end{table}

We present 
 without further comment our duality theorems for $\SM_{2m-1} $
and
%for
 $\CVM{2m}$.  The requisite  facts have  been established  and  no new features arise.

\begin{thm}[Two-sorted Piggyback Strong Duality Theorem for 
$\SM_ {2m-1}$]\label{thm:OddMultimon}  %5.8odd
Fix~$m$ with $m\geq 2$.  Take disjoint sorts $\S^-$ and $\S^+$ each   isomorphic  to $\W_{2m-1}$.
  Let $\Omega_{\S^-} = \{ \delmv\}$ and $\Omega_{\S^+} = \{\delpv\}$.  Consider the two-sorted 
structure based on  $\CM := \{ \S^-, \S^+\} $ having universe
$M_0 : = S^- \du S^+$.  
Let 
\[
G = \{  g, \id_{-+},  \id_{+-}\},\qquad 
H= 
\{   f_1,f_2, \ldots, 
 f_{m-2} \}, \qquad 
K = \{\boldsymbol 0_{\S^-}\}.
\]
\upshape{(}Here $g, f_0,\ldots, f_{m-2}$ are to be interpreted on $\P^-$.\upshape{)} 
 Then $\CMT := (M_0; G, H, K, \Tp) $ strongly dualises $\SM_{2m-1}$.   
\end{thm}

\begin{thm}[Three-sorted Piggyback Strong Duality Theorem for 
$\CVM {2m}$]\label{thm:EvenMultimon}  %5.8
Let  $\S^-$,  $\S^+$ and $\T$ be  disjointified copies of 
$\W_{2m-1}$, $\W_{2m-1}$ and $\W_{2m}$, respectively, with carrier maps $\delmv$, $\delpv$ and $\beta$, respectively.
Let 
$\CM= \{ \S^-, \S^+,\T\}$ and let $\Tp$ be the disjoint union topology 
on  $M_0 := S^- \du S^+\du T$ derived from the discrete
topology on each sort.   Let $
\CMT= (M_0;  G,H,K,\Tp)$, where the elements of $G\cup H$
are the maps presented in Table~{\rm \ref{table:varmon-alterego}} and $K = \{\boldsymbol 0_{\S^-}\}$.  Then $\CMT$ strongly dualises 
$\CVM{2m}$.
\end{thm}

Observe how our elimination of the partial endomorphism $f_0$
from our alter ego for the duality for $\CVA{2m-1}$ allows  
better alignment between the duality theorems for $\CVA{2m}$ and $\CVM{2m}$ than if we had retained it.   %%%%%%%%%%%%%%%

\section{Appendix: The case for multisorted dualities}\label{sec:why-multi}
%In Section~\ref{sec:multisorted} we shall provide a brief formal account
%of multisorted duality theory as we shall apply it in the Sugihara setting.   Here we adopt a more informal style. 
 Here we give a  brief contextual  appraisal of the role of multisorted duality theory, as it applies to distributive-lattice-based algebras and illustrated  by Kleene algebras and lattices. 
We do not include 
in this appendix any results not previously known.  However the material is scattered across a number of sources and the methods 
we want to discuss are often applied to families of classes of algebras, and this can obscure the ideas.
Our  purpose here  is to exemplify through  our running examples
how a multisorted
 approach can be seen as a facilitator, allowing different types of
dual representation to operate collaboratively.  We shall take these ideas forward for Sugihara monoids and algebras in
subsequent papers, in the first instance in~\cite{CPfree}.

%hp1501  This belongs in the Note
%Consider a single-sorted, multi-carrier piggyback duality 
%for a class $\CA= \ISP(\M)$ 
%as compared with an equivalent duality based on a multisorted dual category.  In the former, all the information is encoded 
%in an alter ego with universe $M$.  In the latter the information 
%has been teased apart, so that it is carried on the union of multiple disjoint
%copies of~$M$ indexed by the set of carrier maps (these copies
% provide the sorts of the multisorted alter ego);  the relational structure consists of relations, operations, partial operations and constants, within individual sorts or linking pairs of sorts.  
%This deconstruction process, applicable not just to the alter ego
%itself but to all dual spaces of members of~$\CA$, provides 
%a kit  of components out of which we may hope to build alternative 
%forms of topological relational semantics for~$\CA$.  
%In particular, if $\M$ has a distributive lattice reduct, 
%we may expect to establish a translation to an 
% associated Priestley-style duality and, at the least, thereby 
%to gain  easy access to descriptions 
%of  the lattice reducts of the finitely generated free algebras.   

Quasivarieties generated by small finite distributive lattices with additional operations, with or without distinguished constants,   
 provided a rich source of examples on which Davey and Werner 
were able to test the power of the foundational theory. % in
%\cite{DW83}.  Many of their examples, including $\DM$ and $\class{Kl}$, had affinities with algebraic logic. 
  The emphasis here is on \emph{small}.  Davey and Werner's method relied on their 
NU Duality Theorem:  for a lattice-based quasivariety
$\ISP(\M)$ this yields a dualising alter ego
$(M; \mathbb S(\M^2), \Tp)$.  Taking this forward to arrive at a workable
duality 
necessitated describing \emph{all}  subalgebras of $\M^2$  
 and then
using entailment constructs to weed out superfluous relations.

\begin{ex}[Kleene algebras and lattices: 
hand-crafted single-sorted dualities]\label{Kleene1} 
Let $\boldsymbol 3 = (\{0,a,1\},
\land,\lor, \neg, 0,1\}$,  where $(\{0,a,1\}, \land, \lor)$ is the  distributive lattice with underlying order $0 < a < 1$ and $\neg$ is the involution given by   $\neg 0 = 1$, $\neg 1 = 0$ and $\neg a =a$.  The class~$\Kl$ of \defn{Kleene algebras} is $\CV(\three)$;  it coincides with $\ISP(\three)$.
Moreover, 
$\Kl$ is the unique proper subvariety of the variety $\DM$ of De Morgan algebras which covers Boolean algebras. Specifically, $\Kl$ is characterised within~$\DM$ 
by the law $(a\land \neg a) \land (b \lor \neg b) = a \land \neg a$
(which can equivalently be captured by $a\land \neg a \leq 
b \lor \neg b$).   
The class  of  \defn{Kleene lattices} 
is $\Kla :=\CV(\three_u)	=	\ISP(\three_u)$, where 
$\three_u = (\{0,a,1\},
\land,\lor, \neg)$ is obtained by omitting the constants from the language.

Davey and Werner's duality for $\Kl$ 
has alter ego
$\twiddle{\three}  = (\{0,a,1\};  \preccurlyeq,  -, K_0,\Tp)$,
where 
$\preccurlyeq$ is the  partial order with $()< a$, $1 < a$
and $0$ and $1$  incomparable,  $-$ is 
$\{ 0,a,1\}^2 \setminus \{  (0,1), (1,0)\} $ 
and $K_0=
\{  0,1\}$ \cite{DW83} (or see \cite{CD98}*{4.3.10}).  This was proved by listing all $11$ subalgebras of 
$\three^2$ and whittling down the list down by exhibiting 
entailments.

A minor adaptation yields an alter ego $\twiddle{\three_u}$ for $\three_u$, \textit{viz.}  we add the one-element subalgebra
$\mathbf a$, regarded as a nullary operation,  to~$\twiddle{\three}$.    (A detailed account is given in \cite{FG17}*{Section~4}. See also Example~\ref{ex:Kleene3} below.)
\end{ex}

The number  $3$ is small.
The ad hoc method in Example~\ref{Kleene1} will rarely be viable for quasivarieties  $\ISP(\M)$, 
 either by hand  or even with computer assistance:  the number of subalgebras of $\M^2$ is liable to be unmanageably large.
This led 
Davey and Werner to develop their piggyback method \cite{DW85}.  It applies to \emph{certain}
quasivarieties $\ISP(\M)$, 
where $\M$ is finite, with a distributive lattice reduct (with or without bounds). 
The  idea is to exploit  Priestley duality as it applies to the reducts to identify a dualising alter ego more refined than 
that delivered by the NU Duality Theorem.  
But in its original form 
piggybacking is contingent on there being 
a \emph{single} carrier map (a  homomorphism from (the reduct of) $\M$ into $\two$)   which, composed with the operations of $\M$, 
separates the points of~$\M$.
The separation condition fails for  
Kleene algebras, as explained in the introduction to \cite{DP87}.
It was their analysis of this `rogue' example which
led 
Davey and Priestley to introduce multisorted dualities and to develop a far-reaching generalisation of the piggyback method.  
The latter is 
applicable to any quasivariety or variety generated by a finite  distributive-lattice-based
algebra $\M$.   This subsumes a brute force version for a 
quasivariety $\ISP(\M)$ whereby all homomorphisms from (the reduct of)~$\M$ into $\two$ are employed as carrier maps.  This 
was what we used in~\cite{CPsug}*{Theorem~6.4}.

\begin{ex}[Kleene algebras and lattices:  piggyback dualities with two carrier
maps]\label{ex:Kleene2}
Consider~$\Kl$.  We  
 let $\boldsymbol 3$
be the algebra defined in Example~\ref{Kleene1} and let $
\fnt U (\boldsymbol 3)$ be its  reduct in $\CCD$.   Let $\beta^-$, respectively $\beta^+$,  the $\CCD$-homomorphism from $
\fnt {U} (\boldsymbol 3)$ to $\two$ sending $a$ to~$1$, respectively to 
$0$.   Then
\begin{multline*}
\left.  
\begin{matrix}
(\beta^-,\beta^-)^{-1}(\leq)\\
\\[.5ex]
(\beta^+,\beta^+)^{-1}(\leq) \ \\
(\beta^-,\beta^+)^{-1}(\leq)\ \\
(\beta^+,\beta^-)^{-1}(\leq)\  \\
\end{matrix}
\right\}
\quad \text{has a unique maximal subalgebra which is} \\[-2.5ex]
\left\{
\hbox{\hspace*{-.55cm}
\begin{tabular}{ll}
&$\preccurlyeq$,  the partial order with $0< a$, $1 < a$ and $0$, $1$ \\
 &               \hspace{5cm} 
 incomparable\\[.1ex]
&$\{ 0,a,1\}^2 \setminus \{  (0,1), (1,0)\}$  
\\
&$\{  (0,0),(1,1)\}$\\
&$\succcurlyeq := \preccurlyeq^\partial $  
\end{tabular}
}
\right.
\end{multline*}

We obtain a dualising  alter ego $\twiddle{\three}$ for 
$\boldsymbol 3$ by taking 
the first three of these relations and the discrete topology.
This duality is already strong.
We thereby 
recapture Davey and Werner's original  duality for $\Kl$.

For Kleene lattices, 
the Piggyback Strong Duality Theorem as formulated 
in \cite{CPsug}*{Theorem~3.3}, for single-sorted piggybacking over~$\CCD_u$, applies directly (or see \cite{FG17}*{Section~4}
for a treatment based on the NU Strong Duality Theorem).
\end{ex}

The dualities for $\Kl$ and $\Kla$  recalled in 
Example~\ref{Kleene1} 
 have 
 disadvantages.   
The roles
of the elements $\preccurlyeq$, $-$ and $K_0$ in the alter ego 
are far from obvious, so the dualities lack intuitive appeal, and the
axiomatisation of the dual categories (see \cite{CD98}*{4.3.9} and 
\cite{FG17}*{Proposition~4,3})  does little to assist.  In addition, obtaining  
$\F_{\Kl}(s)$, the free Kleene algebra  on~$s$ generators, from its  natural dual $\twiddle{\boldsymbol 3}^S$ is not the transparent process we could wish.

\begin{ex}[Kleene algebras: a rival duality]\label{ex:Kleene3}
 Cornish and Fowler \cite{CF79}
 proved that 
the category of Kleene algebras is dually equivalent to 
the category $\cat Y$ of \defn{Kleene spaces}.  This provides what would now be referred to as a \defn{restricted Priestley duality}%  
\footnote{The nomenclature is a little ambiguous and the term `enriched' is sometimes used instead of `restricted'.  Often both  adjectives  have a valid connotation,  with `restricted' referring to the forgetful functor
to Priestley spaces not in general being surjective, whereas `enriched'
refers to the dual objects being Priestley spaces with extra structure, usually relations or operations,  with appropriate topological conditions imposed.}; see~\cite{DG17}.
 Here
the objects of
$\cat Y$ are topologised structures $\CY = (Y; \leq ,g,\Tp)$ where 
$(Y; \leq ,\Tp)$ is a Priestley space and~$g$ is a continuous 
order-reversing involution on~$Y$ such that
$\forall y \, \bigl( y \geq g(y) \text{ or } y \leq g(y)\bigr)$;
 the  morphisms are continuous maps preserving $\leq$ and~$g$.  It is an easy exercise in  
Priestley duality to show that  
 the existence of a map~$g$, as in the definition of a Kleene space, is exactly what is needed to capture the Kleene 
negation on the Priestley second dual of the reduct of a  Kleene algebra.  The objective of Cornish and Fowler was to describe 
coproducts, and in particular free algebras, in $\Kl$ by using the duality,  This was successful, but the construction is 
convoluted.   

A version of Cornish and Fowler's duality exists for Kleene lattices, in which the dual category involves pointed Priestley spaces. 
\end{ex}

%hp2002  reworded to get round nasty hyphenation problem
Restricted 
Priestley dualities have the advantage of providing a 
 set-based representation, akin to Kripke  semantics but 
with dual spaces which are relational structures equipped also 
with topology.  
However,  
as  might be expected of second-hand dualities, restricted 
Priestley dualities 
often lack many of the  good features of a natural duality,
 such as 
direct access to free algebras, whose natural dual spaces are given by concrete powers.  
By contrast, Cornish and Fowler's use of their duality  to describe free Kleene algebras is far from transparent.  
For a finitely generated variety or
quasivariety  of distributive-lattice-based algebras,  a natural
duality and a restricted Priestley duality each has a role to play.
 But for duality methods to be at their most powerful we need to
be able translate backwards and forwards between the two 
dual categories.

We now play the multisortedness trump card.  
Moving to a dual category of multisorted structures  will make it
clear how to reconcile the single-sorted natural dualities 
 in~\ref{ex:Kleene2} 
and the restricted Priestley dualities in \ref{ex:Kleene3} once the former are  recast in terms of  dual categories of multisorted structures.

\begin{ex}[Kleene algebras:  two-sorted duality]\label{ex:Kleene4}  
Take disjoint copies of~$\three$, denoted $\three^-$
and $\three^+$. Define  carrier maps as before, allocating $\beta^-$ to $\three^-$ and $\beta^+$ to $\three^+$.  For temporary
convenience, we denote the universe of $\three^\pm$ by
 $M^\pm := \{0^\pm, a^\pm, 1^\pm\}$.  We form an alter ego
%lc0503: I changed the notation. The existing one looked really bad.
%however below in the next example you call this \CMT. I would use $\twiddle{\three^\pm}$   but I leave it to you.
%$\twiddle{\three^-,\three^+}$  
%hp0503  NO WAY!!!   $\twiddle{\three^\pm}$  OMIT SYMBOL FOR IT
$\CMT $ based on the set 
$\{  0^-, a^-,1^-\} \du \{  0^+, a^+,1^+\}$.  We equip 
this with the relations which are maximal subalgebras of 
sublattices $(\gamma, \delta)^{-1}(\leq)$, where 
$\gamma,\delta $ range over $\{\beta^-,\beta^+
\}$.  These binary relations are the same as those in \ref{ex:Kleene2} 
except that the coordinates now need to be tagged to indicate the sorts to which these refer.

\begin{figure}  [ht]
\begin{center}
\begin{tikzpicture}[scale=.5]
\begin{scope}
\node[label=above: {$\scriptstyle (\three^-,\preccurlyeq)$}]  (S-) at  (2,2) {};
\node  (idl-)   at  (4.5, 2.5)  {}; 
\node  (idr-)  at  (5.5, 2.5)  {}; 

\labnode{0,3}{S1-}{label=left:{$ \scriptstyle a^- $}}{.035}{}
%	\node[label=left:{$ \scriptstyle a^- $}] (S1-l) at (0,3) {};
	\labnode{1.5,1.5}{S2-}{label=right:{$\scriptstyle\,1^- $} }{.035}{}
	%\node [label=right:{$\scriptstyle1^- $}] (S2-r) at (1.5,1.5) {};
\labnode{-1.5,1.5}{S3-}{ label=left:{$  \scriptstyle 0^-$}}{.035}{}
	%	\node [label=left:{$  \scriptstyle 0^-$}] (S3-l) at (-1.5,1.5) {};

   		\draw  (S1-) -- (S2-);
		\draw  (S1-) -- (S3-);

	\end{scope}

\begin{scope}
\node[label=below:  {$\scriptstyle (\three^+,\succcurlyeq)$}]  (S-) at  (2,-.3) {};
\node  (idl+)  at  (4.5, -.5) {};
\node  (idr+)  at  (5.5, -.5) {} ;

\labnode{0,-1}{S1+}{label=left:{$\scriptstyle a^+ $}}{.035}{}
	%\node [label=left:{$\scriptstyle a^+ $}] (S1+) at (0,-1) {};
\labnode{1.5,0.5}{S2+}{label=right:{$\scriptstyle\, 1^+$}}{.035}{}
%\node [label=right:{$\scriptstyle 1^+ $}] (S2+r) at (1.5,0.5) {};
\labnode{-1.5,0.5}{S3+}{label=left:{$\scriptstyle 0^+$}}{.035}{}		
%\node [label=left:{$\scriptstyle 0^+$}] (S3+l) at (-1.5,0.5) {};

 \draw  (S1+) -- (S2+);
\draw  (S1+) -- (S3+);

\end{scope}

\draw[thick, dotted] (-1.5,1) ellipse (.35cm and 1cm);
\draw[thick,dotted] (1.5,1) ellipse (.35cm and 1cm);

\begin{scope}[xshift=-10cm]
\node  (idl-)   at  (4.8, 2.5)  {}; 
\node  (idr-)  at  (5.2, 2.5)  {}; 
\node  (idl+)  at  (4.8, -1) {};
\node  (idr+)  at  (5.2, -1) {};

\draw[-latex,dashed] (idl+)  to node [xshift =-3pt, yshift=3pt, swap]  
{$\scriptstyle  \id_{+-} \ \quad $} (idl-);
\draw[latex-,dashed] (idr+)  to node [xshift = 2pt,yshift=-2pt] {$\quad \ \  \scriptstyle \id_{-+} $}   (idr-);
\end{scope}

\end{tikzpicture}
\end{center}
\caption{Two-sorted alter ego for $\Kl$ \label{fig:altegoKl}}
\end{figure}

Figure~\ref{fig:altegoKl} depicts the alter ego, and its relations.

A key feature of the multisorted approach is that duals of
free algebras are calculated `by sorts'.  In the case of 
$\Kl$ this means that the dual of $\F_{\Kl}(s)$, the free Kleene
algebra on $s$ generators, has universe 
$(M^-)^s \du (M^+)^s$, with $\preccurlyeq$, $-$ and $K_0$
lifted pointwise.   
\end{ex} 

Davey and Priestley in \cite{DP87}*{Theorem~3.8} 
%hp2201 ref completed
presented the first 
systematic translation process for toggling between natural
dualities and  their restricted Priestley duality counterparts.  
The context there, and in \cite{CD98}*{Section~7.6} too,  was Ockham algebras::  within a much wider
framework, Kleene algebras 
provide the simplest non-trivial example.  
%hp2201 rephrased 
The results given 
 in  \cite{DP87}*{Section~3}  apply to Ockham algebra quasivarieties.  They 
are heavy on notation  and hide the simplicity of the ideas.  
%%%%%%%%%%%%%%%%%

\begin{ex}[Reconciliations: the best of both worlds]\label{ex:Kleene5}
We don't give a full translation, preferring to illustrate the process as it applies to $\F_{\Kl}(1) $ and $\F_{\Kl}(2)$.  
The dual spaces of these algebras are, respectively, $\CMT$ and 
$\CMT^2$. We may view the union of the multisorted piggyback relations,
\[
\preccurlyeq^- \, \cup \, \succcurlyeq^+ \, \cup \, \{0^-,0^+), (1^-,1^+)\} \,  \cup \, (M^+\times M^-) \setminus
\{ (0^+,1^-), (1^+,0^-)\} 
\]
as a preorder on $M^-\cup M^+$.
When we quotient by  $\preccurlyeq \, \cap\, \succcurlyeq $ to obtain a partial order, 
the each of  the pairs   $0^+$, $0^-$ and $1^-$, $1^+$ 
collapses to a point.  On the resulting poset, $\two^2$, the linking maps
$\id_{-+}$ and $\id_{+-}$ induce a well-defined order-reversing
involution~$g$.   We have  obtained the Kleene space 
which is the Cornish--Fowler dual of $\F_{\Kl}(1)$.  

With $
\CMT^2$ we proceed similarly, remembering that the power
is formed `by sorts'; see Figure~\ref{fig:dualKl2}.  
The natural dual of $\F_{\Kl}(2)$ is shown in 
Figure~\ref{fig:dualKl2}(a).  Quotienting, 
we obtain the 
ordered set~$Y$ shown in~(b).    The isomorphisms $\id_{-+}$ and 
$\id_{+-}$ linking the two sorts encode, after the quotienting, 
the map~$g$ in the restricted Priestley duality.  
The Kleene  example is particularly simple because the map~$g$ is determined by the order structure.  

Denote the quotienting map by~$q$.  Then, in the quotient,
$\{\, y \mid y \geq  g(y)\,\}  = \beta^-(\three^-)$ and  
$\{\, y \mid y \leq  g(y)\,\}  =  \beta^+(\three^+)$.  
 We are looking 
here at the dual spaces of a very particular algebra in 
$\Kl$, but this behaviour is exhibited  quite generally.
This reflects the fact that single-carrier, single-sorted
piggybacking does not work.  See \cite{CD98}*{Section~7.2}, \cite{DP87} and \cite{CPcop}*{Section~3}
for detailed explanations.

Likewise to find the cardinality and the structure of
a free Kleene algebra $\F_{
\Kl}(s)$, a good way to proceed is to form its
2-sorted natural dual and then to pass to the restricted Priestley dual.  

For $\Kla$, the translation works in the same way, but now 
piggybacking is over $\CCD_u$; when recapturing the reduct of
an algebra from its $\CCD_u$  the duality used is that between 
$\CCD_u$  and doubly pointed Priestley spaces, as in \cite{CD98}*{Section~1.2 and Theorem~4.3.2}.
%hp2201  check exact ref.
The lattice reduct of $\F_{\Kl}(s)$ has a join-irreducible top
and meet-irreducible bottom.  Simply delete these elements to
obtain $\F_{\Kla}(s)$.

\begin{figure}  [ht]
\begin{center}
\begin{tikzpicture}[scale=.36]
\begin{scope}[xshift=8.5cm]  %ok
\draw [-latex,decorate,decoration={snake,amplitude=1pt,segment length=5pt,
                pre length=3pt,post length=3pt}]
%wiggly] 
(-2,1) -- (2,1); 
\end{scope}
                                                                                                             %========================

\begin{scope}[xshift=-2.5cm]
\posnode{0,3}{S01-}{.05}
\posnode{1.5,1.5}{S02-}{.035}  % at (1.5,1.5) {}; 
\posnode{-1.5,1.5}{S03-}{.035}

   		\draw (S01-) -- (S02-);
		\draw (S01-) --  (S03-);

\draw[thick,dotted] (-1.5,1) ellipse (.5cm and 1cm);
\draw[thick,dotted] (1.5,1) ellipse (.5cm and 1cm);

	\end{scope}

%%=====================================

\begin{scope}[xshift=2.5cm]
	\posnode{0,3}{S11-} {.035}
		\posnode{1.5,1.5}{S12-}{.035}
		\posnode{-1.5,1.5}{S13-}{.035}

   		\draw  (S11-) -- (S12-);
		\draw  (S11-) -- (S13-);

\draw[thick,dotted] (-1.5,1) ellipse (.5cm and 1cm);
\draw[thick,dotted] (1.5,1) ellipse (.5cm and  1cm);
	\end{scope}
%===========================

\begin{scope}[yshift=1.2cm]
	\posnode{0,3}{Sa1-}{.05}
		\posnode{1.5,1.5}{Sa2-}{.035}
		\posnode{-1.5,1.5}{Sa3-}{.035}

\end{scope}

%%%%%%%%%%%%%%%%%%%%%%%%%%%%%%%%%%%%%%%%
%%%%%%%%%%%%%%%%%%%%%%%%%%%%%%%%%%%%%%

   		\draw (Sa1-) -- (Sa2-);
		\draw (Sa1-) -- (Sa3-);

%================================

%top left end 
%%%%%%%%%%%%%%%%%%%%%%%%%%%%%%%%%%%%%%%%%%%%

%%start bottom left 

\begin{scope}[xshift=-2.5cm]
	\posnode{0,-1}{S01+}{.05} 
		\posnode{1.5,0.5}{S02+}{.035}
		\posnode{-1.5,0.5} {S03+}{.035}

   		\draw (S01+)-- (S02+);
		\draw  (S01+) -- (S03+);

\end{scope}

%=================================

\begin{scope}[xshift =2.5cm]
	\posnode{0,-1}{S11+} {.05}
		\posnode{1.5,0.5}{S12+}{.035}
		\posnode{-1.5,0.5} {S13+}{.035}

   		\draw  (S11+) -- (S12+);
		\draw  (S11+) -- (S13+);

\end{scope}
%==========================

\begin{scope}[yshift=-1.2cm]
	\posnode{0,-1}{Sa1+}{.05} 
		\posnode{1.5,0.5}{Sa2+}{.035}
		\posnode{-1.5,0.5}{Sa3+}{.035}

   		\draw  (Sa1+) -- (Sa2+);
		\draw (Sa1+) -- (Sa3+);

\end{scope}

%\end{scope}
%\end{tikzpicture}
%
%\end{center}
%\end{figure}
%\end{document}

%=======================================

	\draw (S01-)--(Sa1-);
	\draw  (S03-)--(Sa3-);
    \draw  (S02-)--(Sa2-);

	\draw  (S03+)--(Sa3+);
	\draw  (S01+)--(Sa1+);
	\draw  (S02+)--(Sa2+);

	\draw  (S11-)--(Sa1-);
	\draw  (S13-)--(Sa3-);
\draw  (S12-)--(Sa2-);

\draw  (S11+)--(Sa1+);
	\draw (S13+)--(Sa3+);
\draw  (S12+)--(Sa2+);
%%

%

%+++++++++++++++ links between components 

\begin{scope}[xshift=-12.5cm]
\node  (idl-)   at  (4.8, 2.5)  {}; 
\node  (idr-)  at  (5.2, 2.5)  {}; 
\node  (idl+)  at  (4.8, -1) {};
\node  (idr+)  at  (5.2, -1) {};

             \draw[-latex,dashed] (idl+)  to node [xshift =-3pt, yshift=3pt, swap]  
{$\scriptstyle  \id_{+-} \ \quad $} (idl-);
\draw[latex-,dashed] (idr+)  to node [xshift = 2pt,yshift=-2pt] {$\quad \ \  \scriptstyle \id_{-+} $}   (idr-);
\end{scope}   %link maps between sorts 

\node[label=below: {(a) 2-sorted dual}]  (laba) at (0,-2.8) {};

%%%%%%%%%%%%%%%%%%%%%%%%%%%%%%%%%%%%%%%%%%
%%%%%%%%%%%%%%%%%%%%%%%%%%%%%%%%%%%%%%%%%%

%end left part

%%%%%%%%%%%%%%%%%%%%%%%%%%%%%%%%%%%%%%%%%%%%%

\begin{scope}[xshift=16cm]   %start of (b)

\begin{scope}[xshift=1.2cm, yshift=.5cm]   %%g-map  OK

\node  (g-)  at  (5.2, 2.5)  {}; 

\node  (g+)  at  (5.2, -1) {};
\draw[latex-latex,dashed] (g+)  to node [xshift = 2pt,yshift=-2pt] {$\quad \ \  \scriptstyle g$}   (g-);
\end{scope}   %end g-map 

\begin{scope}[xshift=-2.5cm]
	\posnode{0,3}{S01-} {.05}
		\labnode{1.5,1.5}{S02-}{fill=black}{.035}{}
		\labnode{-1.5,1.5}{S03-}{fill=black}{.035}{}

   		\draw  (S01-) -- (S02-);
		\draw  (S01-) -- (S03-);

	\end{scope}

%==============================

\begin{scope}[xshift=2.5cm]
	\posnode{0,3}{S11-}{.05}
		\labnode{1.5,1.5}{S12-}{fill=black}{.035}{}
		\labnode{-1.5,1.5}{S13-}{fill=black}{.035}{}

   		\draw  (S11-) -- (S12-);
		\draw (S11-) -- (S13-);

	\end{scope}  

%===================================

\begin{scope}[yshift=1.2cm]  
	\posnode{0,3}{Sa1-}{.05}
		\posnode{1.5,1.5}{Sa2-}{.035}
		\posnode{-1.5,1.5}{Sa3-}{.035}

   		\draw (Sa1-) -- (Sa2-);
		\draw (Sa1-) -- (Sa3-);

	\end{scope}
   %end of top right part

%%%%%%%%%%%%%%%%%%%%%%%%%%%%%%%%%%%%%%%%%%%%
\begin{scope}[yshift=27pt]  %start  bottom  right 
\begin{scope}[xshift=-2.5cm]
	\posnode{0,-1}{S01+} {.05} 
		\posnode{1.5,0.5}{S02+}{.035}
		\posnode{-1.5,0.5}{S03+}{.035}

   		\draw (S01+) -- (S02+);
		\draw (S01+) -- (S03+);
	
\end{scope}
%-------------------------------------
\begin{scope}[xshift =2.5cm]
	\posnode{0,-1}{S11+}{.05} 
		\posnode{1.5,0.5}{S12+}{.035}
		\posnode{-1.5,0.5}{S13+}{.035}

   		\draw (S11+) -- (S12+);
		\draw (S11+) -- (S13+);

\end{scope}
%====================================

\begin{scope}[yshift=-1.2cm]
	\posnode{0,-1}{Sa1+}{.05}
		\posnode{1.5,0.5}{Sa2+}{.035}
		\posnode{-1.5,0.5}{Sa3+}{.035}

\end{scope}

   		\draw  (Sa1+) -- (Sa2+);
		\draw  (Sa1+) -- (Sa3+);

\end{scope}

	\draw  (S01-)--(Sa1-);
	\draw  (S03-)--(Sa3-);
    \draw  (S02-)--(Sa2-);

	\draw (S03+)--(Sa3+);
	\draw  (S01+)--(Sa1+);
	\draw  (S02+)--(Sa2+);

	\draw  (S11-)--(Sa1-);
	\draw  (S13-)--(Sa3-);
\draw (S12-)--(Sa2-);

\draw  (S11+)--(Sa1+);
	\draw  (S13+)--(Sa3+);
\draw (S12+)--(Sa2+);
\node[label=above: {(b)  Cornish--Fowler dual}] (labb) at  (0,-5.2)
{};  
\end{scope} 

 %end of (b)   --the big shift right 
\end{tikzpicture}
\vspace*{-0.5cm}
\end{center}

\caption{Dual spaces of $\F_{\Kl}(2)$ \label{fig:dualKl2}}
\end{figure}

\end{ex}

 Further examples of translation between dualities of different types are given in \cites{HPLuk,DBlat,CPGo}  and a systematic account of this process, for the underlying lattices,  
is presented in \cite{CPcop}. 
However one would not expect 
 straightforward, or uniform, procedures for capturing  
algebraic operations of arity $> 1$ in dual terms, and for 
translating  between different types of dual representations.  

Discussion of this important aspect of our methodology, 
and its application to finitely generated free Sugihara algebras and monoids,  is deferred to a later paper.

%%%%%%%%%%%%%%%%%%%%%%%%%%%%%%%%%%%%%%%%%%%%%

%%\begin{thebibliography}{99}
%\bibliographystyle{amsplain}
%\bibliography{sugmult-bibfin}
%
%%\end{thebibliography}
\end{document}